\documentclass[11pt]{article}
\usepackage{amsmath} 
\usepackage{graphicx} 
\usepackage{latexsym} 
\usepackage{amsfonts} 
\usepackage{amssymb} 
\usepackage{fancyhdr} 
\usepackage{hyperref} 
\usepackage{xcolor} 
\usepackage{pgfplots} 
\usepackage{amsthm} 
\usepackage{bbm}
\usepackage[nottoc]{tocbibind}
\usepackage{autonum}
\usepackage{calrsfs}

\newtheorem{thm}{Theorem}[section]
\newtheorem{lemma}[thm]{Lemma}
\newtheorem{prop}[thm]{Proposition}

\newtheorem{defn}[thm]{Definition}
\newtheorem{claim}[thm]{Claim}

\newtheorem{coro}[thm]{Corollary}
\newtheorem*{thm:3.6.2}{Theorem \ref{thm:3.6.2}}
\newtheorem*{prop:nonconstantweights}{Theorem \ref{prop:nonconstantweights}}

\newcommand{\hcap}{\text{hcap}}

\author{Andrew Starnes}
\title{The Loewner Equation for Multiple Hulls}

\begin{document}

\maketitle

\begin{abstract}
\noindent Kager, Nienhuis, and Kadanoff conjectured that the hull generated from the Loewner equation driven by two constant functions with constant weights could be generated by a single rapidly and randomly oscillating function. We prove their conjecture and generalize to multiple continuous driving functions. In the process, we generalize to multiple hulls a result of Roth and Schleissinger that says multiple slits can be generated by constant weight functions. The proof gives a simulation method for hulls generated by the multiple Loewner equation. 
\end{abstract}

\section{Introduction}

The Loewner equation is the initial value problem
\begin{equation}\label{eqn:LEintro}
    \frac{\partial}{\partial t} g_t(z)
    =\frac{2}{g_t(z)-\lambda(t)},
    \quad g_0(z)=z.
\end{equation}
where $\lambda:[0,T]\to\mathbb{R}$ is called the driving function. For $z\in\mathbb{H}$, a solution exists up to a maximum time, call it $T_z$. The collection of points
\begin{equation}
    K_t=\left\{z\in\mathbb{H}:T_z\leq t\right\}
\end{equation}
is called a hull. A fundamental note is that there is a one-to-one correspondence between hulls and driving functions. The map $g_t$ in (\ref{eqn:LEintro}) is a conformal map from $\mathbb{H}\setminus K_t$ to $\mathbb{H}$ (see Section \ref{loewnerequation} for more details).
The Loewner equation was discovered in 1923 by Charles Loewner in pursuit of proving the Bieberbach conjecture and it reemerged in 2000, when Oded Schramm discovered its relationship to the scaling limit of loop-erased random walks. This discovery lead to construction of the Schramm-Loewner Evolution (SLE$_\kappa$) and has been vigorously studied ever since.

In this paper, our main focus is the multiple Loewner equation
\begin{equation}\label{eqn:multiLEr}
    \frac{\partial}{\partial t} g_t(z)=\sum_{k=1}^{n}\frac{2w_k(t)}{g_t(z)-\lambda_k(t)}\text{ a.e. }t\in [0,T],
    \quad g_0(z)=z
\end{equation}
where $\lambda_1,...,\lambda_n:[0,T]\to\mathbb{R}$ are continuous and $w_1,...,w_n\in L^1[0,T]$ are weight functions. In \cite{KNK04}, it was conjectured that the multiple Loewner equation driven by $\lambda_1=-1$ and $\lambda_2=1$ with constant weights equal to $\frac{1}{2}$ could be realized by a single rapidly and randomly oscillating function driven by the Loewner equation (\ref{eqn:LEintro}). We prove this conjecture with the following more general result.

\begin{prop}\label{prop:nonconstantweights}
Let $K=\bigcup_{i=1}^{n}K_i$, where $K_1,...,K_n$ are disjoint hulls driven by continuous driving functions in the chordal sense. Then $K$ is the limit of hulls generated by a sequence of randomly and rapidly oscillating functions.
\end{prop}

\noindent This proposition inspires a simulation method for hulls from the multiple Loewner equation driven with constant weights. The idea is to use a single driving function that randomly and rapidly oscillates between the multiple driving functions, which generalizes the conjecture in \cite{KNK04}. We simulate the hull investigated in \cite{KNK04} and compare it to the actual hull in Section \ref{simulations}.

The proof of Proposition \ref{prop:nonconstantweights} result follows from a generalization of Theorem 1.1 in \cite{RS14}, which says that multiple slits can be generated through the multiple Loewner equation by continuous driving functions and constant weights. We generalize this to multiple hulls, as follows:

\begin{thm}\label{thm:3.6.2}
Let $K^1,...,K^n$ be disjoint Loewner hulls. Let $\hcap(K^1\cup\cdots\cup K^n)=2T$. Then there exist constants $w_1,...,w_n\in(0,1)$ with $\sum_{k=1}^{n}w_k=1$ and continuous driving functions $\lambda_1,...,\lambda_n:[0,T]\to\mathbb{R}$ so that
\begin{equation}
    \frac{\partial}{\partial t} g_t(z)=\sum_{k=1}^{n}\frac{2w_k}{g_t(z)-\lambda_k(t)},\quad g_0(z)=z
\end{equation}
satisfies $g_T=g_{K^1\cup\cdots\cup K^n}$.
\end{thm}

One significant difference between Theorem 1.1 in \cite{RS14} and this result is the lack of uniqueness. This is due to the fact that we do not know the growth over time of the hulls in Theorem \ref{thm:3.6.2}, we only know what the hull looks like at a particular time. This ambiguity allows the possibility that a hull can be driven by different driving functions, whereas any slit has a unique driving function. For example, if the hull is a semi-circle of radius 1 centered at 0, then two ways to generate this hull are by travelling the boundary clockwise or counterclockwise. This corresponds to scaling the driving function by $-1$. However, if we have $K_t^j$ for each time and each $j\in\{1,...,n\}$, then using the same proof of uniqueness for slits from \cite{RS14}, we would have uniqueness in the multiple hull setting as well.

This paper is structured as follows: Section \ref{convofrapidrandomhulls} introduces enough about the Loewner equation to prove Proposition \ref{prop:nonconstantweights} from Theorem \ref{thm:3.6.2}. Section \ref{simulations} discusses simulation of the multiple Loewner equation. Section \ref{background} rigorously covers the background information about the Loewner equation, hulls, and a generalization of the tip of a curve, which is needed to prove Theorem \ref{thm:3.6.2}. Finally, Section \ref{precompactness} gives the proof of Theorem \ref{thm:3.6.2}. Sections \ref{background} and \ref{precompactness} can be read without reading Sections \ref{convofrapidrandomhulls} and \ref{simulations}. As in \cite{RS14}, we will only show results for $n=2$ and the general result follows from mathematical induction.\par\vspace{12pt}

\noindent\textbf{Acknowledgement:} I would like to thank Joan Lind for all of her help and support with this paper.

\section{Convergence of Hulls Using Rapid and Random Oscillation}\label{convofrapidrandomhulls}
\subsection{Brief Introduction to Loewner Equation}

Our goal is to discuss convergence of a rapidly and randomly oscillating driving function, but we need to define what convergence we will use. We say that $g_t^n$ converges to $g_t$ in the Carath\'eodory sense, denoted $g_t^n\xrightarrow{Cara} g_t$, if for each $\epsilon>0$ $g_t^n$ converges to $g_t$ uniformly on the set
\begin{equation}
    [0,T]\times\{z\in\mathbb{H}:\text{dist}(z,K_t)\geq\epsilon\}.
\end{equation}
This form of convergence allows for convergence of functions when their domains are changing.

\subsection{Introduction to Conjecture}\label{introtoconjecture}

In Section 6 of \cite{KNK04}, Kager, Nienhuis, and Kadanoff investigate the multiple Loewner equation generated from constant driving functions, $\lambda_1\equiv -1$ and $\lambda_2\equiv 1$, and constant weights, $w_1=w_2=\frac{1}{2}$. They show that the hull is given by
\begin{equation}\label{eqn:knkhull}
K_t=\left\{\sqrt{\frac{2\theta_t}{\sin(2\theta_t)}}(\pm\cos\theta_t+i\sin\theta_t)\right\}
\end{equation}
where $\theta_t$ increases from 0 to $\frac{\pi}{2}$ as $t$ increases. They make the conjecture that the same hull can be generated by a single driving function that ``makes rapid (random) jumps between the values $\lambda_j$.'' In this section, we will say that a sequence of driving functions generate a hull if the corresponding conformal maps from the Loewner equation converge in the Carath\'eodory sense to the conformal map corresponding to the hull. We will prove their conjecture constructively. The key tool in the proof is the use of the following theorem by Roth and Schleissinger from \cite{RS14} which we use to relate the multiple Loewner equation and a single driving function.

\begin{thm}[2.4 \cite{RS14}]\label{thm:rs2.4rrodf}
For $j\in\{1,2\}$ let $w_j^n,w_j\in L^1[0,1]$ be weight functions and let $\lambda_j^n,\lambda_j\in C[0,1]$ be driving functions with associated Loewner chains $g_t^n, g_t$. If $\lambda_j^n$ converges to $\lambda_j$ uniformly on $[0,1]$ and if $w_j^n$ converges weakly in $L^1[0,1]$ to $w_j$ for $j=1,2$, then $g_t^n$ converges in the Carath\'eodory sense to the chain $g_t$.
\end{thm}
The idea to constructing a randomly, rapidly oscillating driving function is to use the driving functions that generate the hull $K_t$ from the multiple Loewner equation. We do this by dividing up the time interval into smaller intervals and then randomly pick which driving function to use on each small interval. This random picking is governed by the weights. Furthermore, this construction is not limited to the case described above that is considered in \cite{KNK04}. In fact, Proposition \ref{prop:nonconstantweights} is a more general answer to their conjecture.

\subsection{Controlled Oscillation}

Before we tackle the conjecture, we will do an example. In the situation of \cite{KNK04}, let $\lambda_1\equiv -1$, $\lambda_2\equiv 1$, $w_1=w_2=\frac{1}{2}$, and $K_t$ be as in (\ref{eqn:knkhull}). We will create a sequence of rapidly oscillating functions that generate $K_t$. The idea here is essentially the idea in the more general case: divide the interval into smaller pieces and decide whether the driving function is $-1$ or $1$ on each piece. Here, since $w_1=w_2=\frac{1}{2}$, we will simply rotate between the driving functions $-1$ and $1$. Let
\begin{equation}
\lambda^n(t)=\sum_{k=0}^{2^{n-2}}\chi_{[\frac{2k+1}{2^n},\frac{2(k+1)}{2^n})}-\chi_{[\frac{2k}{2^n},\frac{2k+1}{2^n})}(t).
\end{equation}
So, we take $[0,1]$ and divide it into an even number of intervals of the from $[\frac{j}{2^n},\frac{j+1}{2^n})$. When $j$ is even $\lambda^n|_{[\frac{j}{2^n},\frac{j+1}{2^n})}\equiv -1$ and when $j$ is odd $\lambda^n|_{[\frac{j}{2^n},\frac{j+1}{2^n})}\equiv 1$. This means for any $n\in\mathbb{N}$ $\lambda^n(t)=-1=\lambda_1$ for half of the time and $\lambda^n(t)=1=\lambda_2$ for the other half of the time, corresponding to $w_1=w_2=\frac{1}{2}$. Now, we will show that $K_t$ is generated by $\lambda^n$. The proof uses Theorem \ref{thm:rs2.4rrodf} to relate the multiple Loewner equation to a single driving function. We have already defined the driving function, so we will now set up the multiple Loewner equation situation. Define the weight functions
\begin{equation}\label{eqn:controlledjumps}
w_1^n(t):=\sum_{k=0}^{2^{n-1}-1}\chi_{[\frac{2k}{2^n},\frac{2k+1}{2^n})}(t)\quad
\text{and}
\quad
w_2^n(t):=\sum_{k=1}^{2^{n-1}}\chi_{[\frac{2k-1}{2^n},\frac{2k}{2^n})}(t).
\end{equation}
At any time, they sum to 1 and they are never 1 at the same time. We will show $w_j^n$ converges to $\frac{1}{2}$ weakly. Since the conformal maps from the Loewner equation driven by $\lambda^n$ and the conformal maps from the multiple Loewner equation driven by $\lambda_1$, $\lambda_2$, $w_1^n$, and $w_2^n$ are the same, we will have that $K_t$ is generated by $(\lambda^n)_{n=1}^{\infty}$.

\begin{lemma}\label{lem:onehalfweights}
As $n\to\infty$, $w_j^n$ converges weakly to $\frac{1}{2}$ for $j=1,2$ - that is, for each $h\in L^{\infty}[0,1]$
\begin{equation}
\int w_1^n h\to \int\frac{1}{2}h\text{ as }n\to\infty.
\end{equation}
\end{lemma}

\begin{proof}
We will prove this for $j=1$ first. Let $\epsilon>0$ and $h\in L^{\infty}[0,1]$. By Lusin's Theorem there exists $E\in\mathcal{B}([0,1])$ (the Borel sets of $\mathbb{R}$) compact with $m([0,1]\setminus E)<\frac{\epsilon}{2||h||_{\infty}}$ ($m$ denotes Lebesgue measure) and $h$ is continuous on $E$. So,
\begin{equation}
\left|\int_{[0,1]\setminus E}h\left(w_1^n-\frac{1}{2}\right)\right|<\frac{\epsilon}{2}.
\end{equation}
Since $E$ is compact, $h$ is uniformly continuous on $E$. So there exists $\delta>0$ such that for each $x,y\in E$ with $|x-y|<\delta$, we have that $|h(x)-h(y)|<\epsilon$. Also, there exists $N\in\mathbb{N}$ such that for all $n\geq N$, $\frac{1}{2^{n-1}}<\delta$. Let $n\geq N$. For $k\in\mathbb{N}$, define
\begin{equation}
I_k=\left[\frac{k}{2^n},\frac{k+1}{2^n}\right)\cap E.
\end{equation}
Then
\begin{equation}
    \left|\int_{E}h\cdot\left(w_1^n-\frac{1}{2}\right)\right|
    =\left|\sum_{k=0}^{2^n-1}\int_{I_k}\dfrac{(-1)^k}{2}h\right|
    \leq\frac{1}{2}\sum_{k=0}^{2^{n-1}-1}\left|\int_{I_{2k}}h-\int_{I_{2k+1}}h\right|.
\end{equation}
Since the length of $I_{2k}\cup I_{2k+1}$ is $\frac{1}{2^{n-1}}<\delta$, for all $x\in I_{2k}\cup I_{2k+1}$,
\begin{equation}
    h\left(\frac{2k+1}{2^n}\right)-\epsilon
    \leq h(x)
    \leq h\left(\frac{2k+1}{2^n}\right)+\epsilon.
\end{equation}
So,
\begin{equation}
    \left|\int_{I_{2k}}h-\int_{I_{2k+1}}h\right|
    \leq\frac{1}{2^n}(2\epsilon)
    =\frac{\epsilon}{2^{n-1}}.
\end{equation}
Hence,
\begin{equation}
    \left|\int_{E}h\cdot\left(w_1^n-\frac{1}{2}\right)\right|
    \leq\frac{1}{2}\sum_{k=0}^{2^{n-1}-1}\frac{\epsilon}{2^{n-1}}
    <\epsilon.
\end{equation}
This shows that $w_1^n$ converges weakly to $\frac{1}{2}$.\par
Since $w_2^n=1-w_1^n$, we have that $w_2^n$ converges weakly to $\frac{1}{2}$, as well.
\end{proof}

Since $\lambda^n(t)=w_1^n(t)\lambda_1(t)+w_2^n(t)\lambda_2(t)$, by Theorem \ref{thm:rs2.4rrodf}, we have that $K_t$ is generated by $\lambda^n$. This proves that $K_t$ is generated by a rapidly oscillating function.

\subsection{Rapid, Random Oscillation}\label{rapidrandomoscillation}

Now that we have shown that a rapidly oscillating function can be used to satisfy the conjecture in \cite{KNK04}, we turn to proving that we do not have to control the oscillation as we did before. In the random case, we begin construction of the sequence of driving functions by defining weight functions. Let $w_1\in(0,1)$ and $w_2=1-w_1$ be constants. For each $k\in\mathbb{N}$, let $X_k$ be a random variable such that $P(X_k=1)=w_1$ and $P(X_k=0)=w_2$ (i.e. $X_k$ is a Bernoulli random variable). For each $n\in\mathbb{N}$ and $k\in\{1,...,n\}$, define
\begin{equation}
I_k^n=\left[\frac{k-1}{n},\frac{k}{n}\right).
\end{equation}
For each $n\in\mathbb{N}$, define
\begin{equation}\label{eqn:defofrandomweights}
w_1^n=\sum_{k=1}^{n}X_k\chi_{I_{k}^{n}}(t)\quad
\text{and}\quad
w_2^n=\sum_{k=1}^{n}(1-X_k)\chi_{I_{k}^{n}}(t).
\end{equation}
Then for every $t\in[0,1]$ and $n\in\mathbb{N}$, $w_1^n(t)+w_2^n(t)=1$ a.s. Further, $w_1^n(t)=1$ only when $w_2^n(t)=0$ and vice versa. Let
\begin{equation}
\lambda^n(t)=w_1^n(t)\lambda_1(t)+w_2^n(t)\lambda_2(t).
\end{equation}
For any $n\in\mathbb{N}$, $\lambda^n$ rapidly (for large $n$) and randomly oscillates between the values of $\lambda_1$ and $\lambda_2$. The idea here is that $w_j^n$ turns off and on $\lambda_j$. So, essentially we are using the single Loewner equation to approximate the multiple Loewner equation and the weights control which function is turned on or picked in the intervals $I_k^n$. We will first show that $w_j^n$ converges weakly to $w_j$ for $j=1,2$. Then using Theorems \ref{thm:rs2.4rrodf} and \ref{thm:3.6.2}, we will obtain the desired result.

\begin{lemma}\label{lem:randomwweights}
As $n\to\infty$, almost surely $w_j^n$ as in (\ref{eqn:defofrandomweights}) converges weakly to $w_j$ for $j=1,2$.
\end{lemma}

\noindent We will prove this for $j=1$ using a standard approach by proving that convergence holds on intervals, for step functions, for non-negative functions, and for $L^{\infty}$ functions. Then the result will also hold for $j=2$ as $w_2^n=1-w_1^n$.

\begin{claim}
Let $J\subseteq[0,1]$ be an interval. Then almost surely
$\int_J w_1^n\to\int_Jw_1=w_1m(J)$
\end{claim}\par

\begin{proof}
Let $\epsilon>0$ and $J\subseteq [0,1]$ be an interval. Then there exists $N_1\in\mathbb{N}$ such that for all $n\geq N_1$ there exists $a_n\in\{1,...,n\}$ and $m_n\in\{0,...,n-a_n\}$ such that
$\bigcup_{k=a_n}^{a_n+m_n}I_k^n\subseteq J.$
Then there exists a natural number $N_2\geq N_1$ such that for all $n\geq N_2$
\begin{equation}
I_n=\bigcup_{k=a_n}^{a_n+m_n}I_k^n\subseteq J \quad\text{and}\quad m(J\setminus I_n)<\frac{\epsilon}{2}.
\end{equation}
So,
\begin{equation}
\left|\int_{J\setminus I_n}w_1^n-w_1\right|
\leq\left|\int_{J\setminus I_n}dt\right|
=m(J\setminus I_n)
<\frac{\epsilon}{2}
\end{equation}
As $n\to\infty$, $m_n\to\infty$. By the Strong Law of Large Numbers, we have
\begin{equation}
\sum_{k=a_n}^{a_n+m_n}\frac{X_k}{m_n}\to w_1 \text{ a.s.}
\end{equation}
So, there exists $N\geq N_2$ such that for all $n\geq N$
\begin{equation}
\left|\sum_{k=a_n}^{a_n+m_n}\frac{X_k}{m_n}-w_1\right|<\frac{\epsilon}{2}\text{ a.s.}
\end{equation}
Fix $n\geq N$. Then with probability 1, since $m(I_n)=\frac{1}{n}$,
\begin{equation}
\left|\int_{I_n} w_1^n-w_1\right|
=\left|\frac{m_n}{n}\sum_{k=a_n}^{a_n+m_n}\frac{X_k-w_1}{m_n}\right|
=m(I_n)\left|\sum_{k=a_n}^{a_n+m_n}\frac{X_k}{m_n}-w_1\right|
\leq\frac{\epsilon}{2}
\end{equation}
Therefore, as $n\to\infty$, almost surely
\begin{equation}
\int_J w_1^n\to w_1m(J).
\end{equation}
\end{proof}

\begin{claim}
Let $h\in L^{\infty}[0,1]$ be a step function. Then almost surely
\begin{equation}
\int_{[0,1]}h w_1^n\to w_1\int_{[0,1]}h.
\end{equation}
\end{claim}

\begin{proof}
Since $h$ is a bounded step function, there exist finitely many nonempty intervals $J_1,...,J_n$ and $\alpha_1,...,\alpha_n\in\mathbb{R}\setminus\{0\}$ so that $h=\sum_{i=1}^{n}\alpha_i\chi_{J_i}$. Then, by the previous claim, there exists $N$ such that for all $n\geq N$ almost surely
\begin{equation}
\left|\int_{J_i}w_1^n-w_1m(J_i)\right|<\frac{\epsilon}{2\sum_{i=1}^{n}|\alpha_i|}.
\end{equation}
Then with probability 1,
\begin{equation}
\left|\int_{[0,1]}h(w_1^n-w_1)\right|
\leq\sum_{i=1}^{n}\left|\alpha_i\int_{J_i}(w_1^n-w_1)\right|
\leq\sum_{i=1}^{n}|\alpha_i|\frac{\epsilon}{2\sum_{i=1}^{n}|\alpha_i|}
<\epsilon
\end{equation}
This proves the claim.
\end{proof}

\begin{claim}
For $h\in L^{\infty}[0,1]$ with $h\geq 0$, almost surely
\begin{equation}
\int h w_1^n\to w_1\int h.
\end{equation}
\end{claim}

\begin{proof}
Let $h\in L^{\infty}[0,1]$ with $h\geq 0$. Then there exists a step function $f\in L^{\infty}[0,1]$ such that $||f-h||_2\leq\frac{\epsilon}{2}$, where $\|\cdot\|_{k}$ denotes the $L^k[0,1]$ norm. Then there exists $N\in\mathbb{N}$ such that for all $n\geq N$, almost surely $|\int(w_1^n-w_1)|<\frac{\epsilon}{2(\|f\|_{\infty}\vee 1)}$. Also, since $0\leq w_1^n(t)\leq 1$ a.s., $|w_1^n-w_1|\leq 1$ a.s. for all $t\in[0,1]$. So,
\begin{equation}
\left|\int h(w_1^n-w_1)\right|
\leq\left|\int f(w_1^n-w_1)\right|+\left|\int (h-f)(w_1^n-w_1)\right|
\leq\frac{\epsilon}{2}+||h-f||_2
<\epsilon
\end{equation}
This proves the claim.
\end{proof}

\begin{claim}\label{cla:randomweightsconvtow}
For $h\in L^{\infty}[0,1]$, almost surely
\begin{equation}
\int h w_1^n\to w_1\int h.
\end{equation}
\end{claim}

\begin{proof}
Let $h\in L^{\infty}[0,1]$. Then $h^+,h^-\in L^{\infty}[0,1]$ (where $h^+,h^-\geq 0$ and $h=h^+-h^-$). Then there exists $N\in\mathbb{N}$ such that for all $n\geq N$, almost surely
\begin{equation}
\left|\int_{[0,1]}h^+\left(w_1^n-w_1\right)\right|<\frac{\epsilon}{2}
\text{ and }
\left|\int_{[0,1]}h^-\left(w_1^n-w_1\right)\right|<\frac{\epsilon}{2}.
\end{equation}
Then with probability 1,
\begin{equation}
\left|\int h\left(w_1^n-w_1\right)\right|
\leq\left|\int h^+\left(w_1^n-w_1\right)\right|+\left|\int h^-\left(w_1^n-w_1\right)\right|
<\epsilon.
\end{equation}
\end{proof}

\begin{proof}[Proof of Lemma \ref{lem:randomwweights}]
By Claim \ref{cla:randomweightsconvtow}, we have that $w_1^n$ converges weakly to $w_1$. Then as $w_2^n=1-w_1^n$, we have $w_2^n$ converges weakly to $1-w_1=w_2$. So we have the result.
\end{proof}

\begin{proof}[Proof of Proposition \ref{prop:nonconstantweights}]
Apply Theorem \ref{thm:3.6.2} to get $\lambda_1,...,\lambda_n$ continuous functions and constant weights $w_1,..., w_n\in(0,1)$. Applying Lemma \ref{lem:randomwweights} to $w_j^n$ from (\ref{eqn:defofrandomweights}), we have that $w_j^n$ converges weakly to $w_j$ in $L^1[0,1]$ for $j=1,2$. Now, using Theorem \ref{thm:rs2.4rrodf}, we have that we get the convergence we desire.
\end{proof}

\section{Simulating the Multiple Loewner Equation}\label{simulations}

The Loewner equation yields a conformal map that takes sets in the upper half-plane and maps them down to the real line and for this reason is sometimes referred to as the downward Loewner equation. For a map that does the opposite, we can consider the initial value problem
\begin{equation}\label{eqn:upwardle}
    \partial_t f_t(z)
    =\frac{-2}{f_t(z)-\xi(t)},
    \quad f_0(z)=z.
\end{equation}
We call this the upward Loewner equation and the conformal maps $f_t$ grow sets in the upper half-plane. There is a relationship between the downward and upward Loewner equations. If $g_t$ is the map given by the downward Loewner equation driven by $\lambda:[0,T]\to\mathbb{R}$ and $f_t$ is the map given by the upward Loewner equation driven by $\xi(t)=\lambda(T-t)$, then $f_T=g_T^{-1}$.

The idea of the standard algorithm to simulate the hulls from the Loewner equation uses the upward Loewner equation driven by constant functions (see for instance \cite{bauer}, \cite{kennedy07}, \cite{kennedy09}, or \cite{mr}). For a constant driving function $\xi(t)=c$, the solution to the upward Loewner equation is
\begin{equation}\label{eqn:upwardconstant}
    f_t^c(z)=\sqrt{(z-c)^2-4t}+c.
\end{equation}

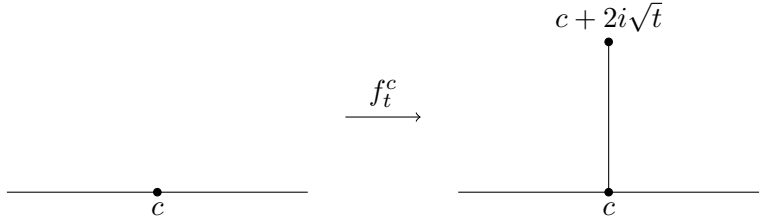
\begin{figure}
\centering
\begin{tikzpicture}

\draw (0,0) -- (4,0);
\draw[fill] (2,0) circle [radius=0.05];
\node[below] at (2,0) {$c$};
\draw[->] (4.5,1) to [out=0,in=180] (5.5,1);
\node[above] at (5,1) {$f_t^c$};

\draw (6,0) -- (10,0);
\draw[fill] (8,0) circle [radius=0.05];
\node[below] at (8,0) {$c$};
\draw (8,0)--(8,2);
\draw[fill] (8,2) circle [radius=0.05];
\node[above] at (8,2) {$c+2i\sqrt{t}$};

\end{tikzpicture}
    \caption{Mapping Up Hull Corresponding to $f_t^c$}
    \label{fig:constantsim}
\end{figure}
\noindent The algorithm for simulating the hull driven by $\lambda:[0,T]\to\mathbb{R}$ with $N+1$ sample points is as follows:
\begin{itemize}
    \item[0.] Compute $\lambda(T)$ and add to hull
    \item[1.] Apply (\ref{eqn:upwardconstant}) with $c=\lambda(T\cdot\frac{N-k}{N})$ to points in hull
    \item[2.] Add $\lambda(T\cdot\frac{N-k}{N})$ to hull
    \item[3.] Repeat steps 1-2 for $k\in\{1,...,N\}$
\end{itemize}

For the multiple Loewner equation, we want to use the same idea as above but our driving function (randomly) oscillates between the driving functions. This is in effect what the proof in Section \ref{rapidrandomoscillation} does to generate the hulls. Let $\lambda_1,\lambda_2:[0,T]\to\mathbb{R}$ be driving functions and $w_1,w_2\in[0,1]$ be constant weights. For $k\in\{0,...,N\}$:

\begin{itemize}
    \item[1.] (Randomly) assign $j_k$ to be either $1$ or $2$ so that $P(j_k=1)=w_1$ and $P(j_k=2)=w_2$
    \item[2.] Define $\lambda(T\cdot\frac{k}{n})=\lambda_{j_k}(T\cdot\frac{k}{n})$
    \item[3.] Repeat steps in previous algorithm
\end{itemize}

We will investigate this algorithm by revisiting the example done in \cite{KNK04} and mentioned here in Section \ref{introtoconjecture} that motivates all of our results. Let $\lambda_1=-1$, $\lambda_2=1$, and $w_1=\frac{1}{2}=w_2$. Recall the hull is given by
\begin{equation}
K_t=\left\{\sqrt{\frac{2\theta_t}{\sin(2\theta_t)}}(\pm\cos\theta_t+i\sin\theta_t)\right\}.
\end{equation}

First, we will control the oscillation by assigning $j_k$ to be 1 when $k$ is odd and 2 when $k$ is even. The simulations for 1,000 and 10,000 oscillations are given in Figures \ref{fig:1000controlledoscillations} and \ref{fig:10000controlledoscillations}. For 1,000 oscillations, the simulated data points are extremely close to the curve. There is a larger spread in the points near the real line since the growth of $f_t^c$ is faster there. For 10,000 oscillations, the simulated data is almost indistinguishable from the curve.

The errors (that is, the maximum distance the data is from the hull) for 1000, 500, 400, 300, 200, 100, 90, 80, 70, 60, 50, 40, 30, 20, 10 controlled oscillations are shown in Figure \ref{fig:errors}, where the blue points correspond to points on the left side (i.e. associated with $\lambda_1$) and the red points correspond to points on the right side (i.e. associated with $\lambda_2$).
Since the last map used in each controlled simulation is $f_t^1$, all of the right sided points are shifted up from their previous positions. This causes more error for these points. On the other hand, the map shifts the left sided points towards the right and reduces the error for these points.
One amazing note is that even for 10 oscillations (11 data points), the error is small enough that simulated points are closer to their respective side than the opposite side (that is, their real parts are on the same side of 0 as their corresponding driving function). Further, for any number of oscillations ($\geq 10$), we could thicken each side of the hull by the error and they would not intersect (up to $T=10$).

Second, we switch to randomly oscillating the driving function. We randomly assign $j_k$ to be 1 or 2 by flipping a fair, virtual coin. In each of Figures \ref{fig:1000randomoscillations} and \ref{fig:10000randomoscillations} are 10 simulated hulls (non-black curves) with 1,000 and 10,000 oscillations (respectively) and the hull (black curves). For 1,000 oscillations, the simulated hulls have the same overall shape (e.g. they approach each other as their imaginary parts increase), but there is significant variation between the curves. For 10,000 oscillations, the simulated hulls are significantly closer to the hull, but there is still variation between the curves. The upshot is that the random hulls are visually a good replacement for the actual hull. Figure \ref{fig:100errors} gives a histogram of 100 simulations of 1,000 random oscillations where left and right sides correspond to the colors blue and red as before.

It appears that the controlled oscillation (i.e. forcing a switch between driving functions) always outperforms the random oscillation. This intuitively makes sense. Say we grow the -1 hull first using $f_t^{-1}$. If we use $f_t^1$ next, the hull corresponding to -1 will be shifted to the right. Instead, if we use $f_t^{-1}$ next, the hull corresponding to -1 will be higher. In the random oscillation case, either of these maps could be used over and over before switching. This would cause the hulls to be higher or more to the left or right than the actual hull. The forced oscillation appears to not allow either side of the hull to get too far away from the actual hull.

\begin{figure}
\begin{minipage}{0.45\textwidth}
    \centering
    \includegraphics[width=0.9\textwidth,height=2in]{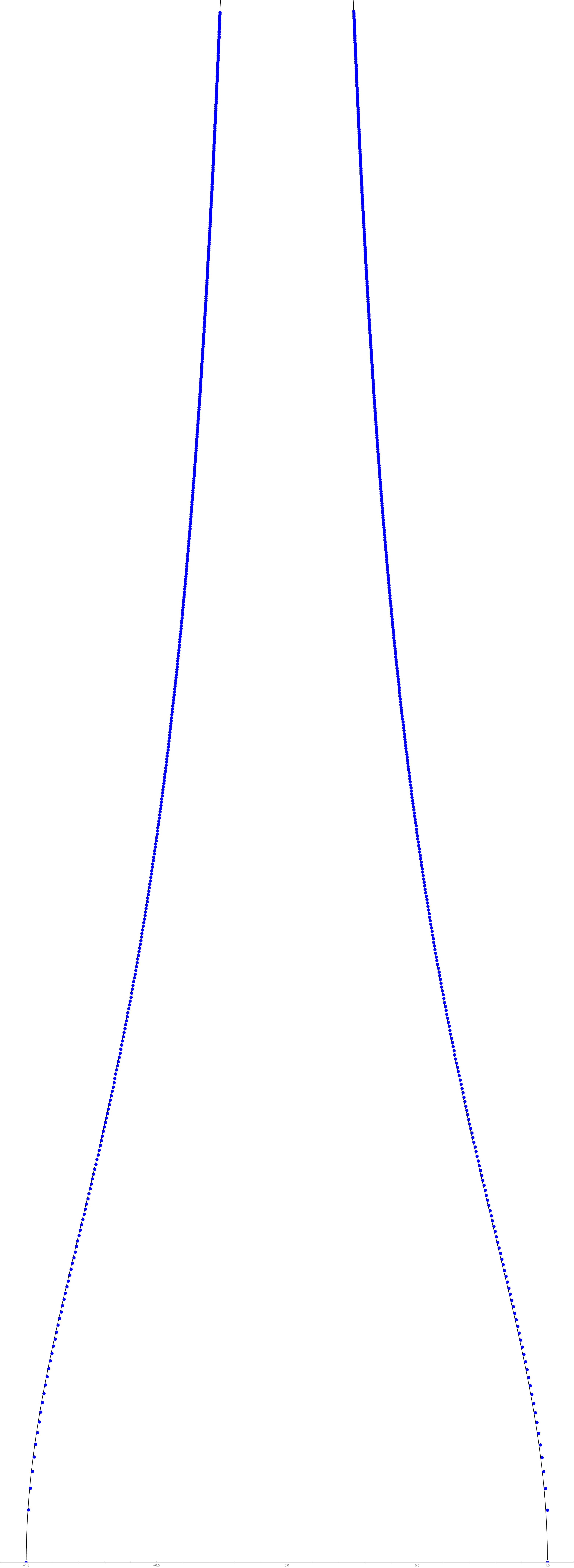}
    \caption{1000 Controlled Oscillations}
    \label{fig:1000controlledoscillations}
\end{minipage}\hfill
\begin{minipage}{0.45\textwidth}
    \centering
    \includegraphics[width=0.9\textwidth,height=2in]{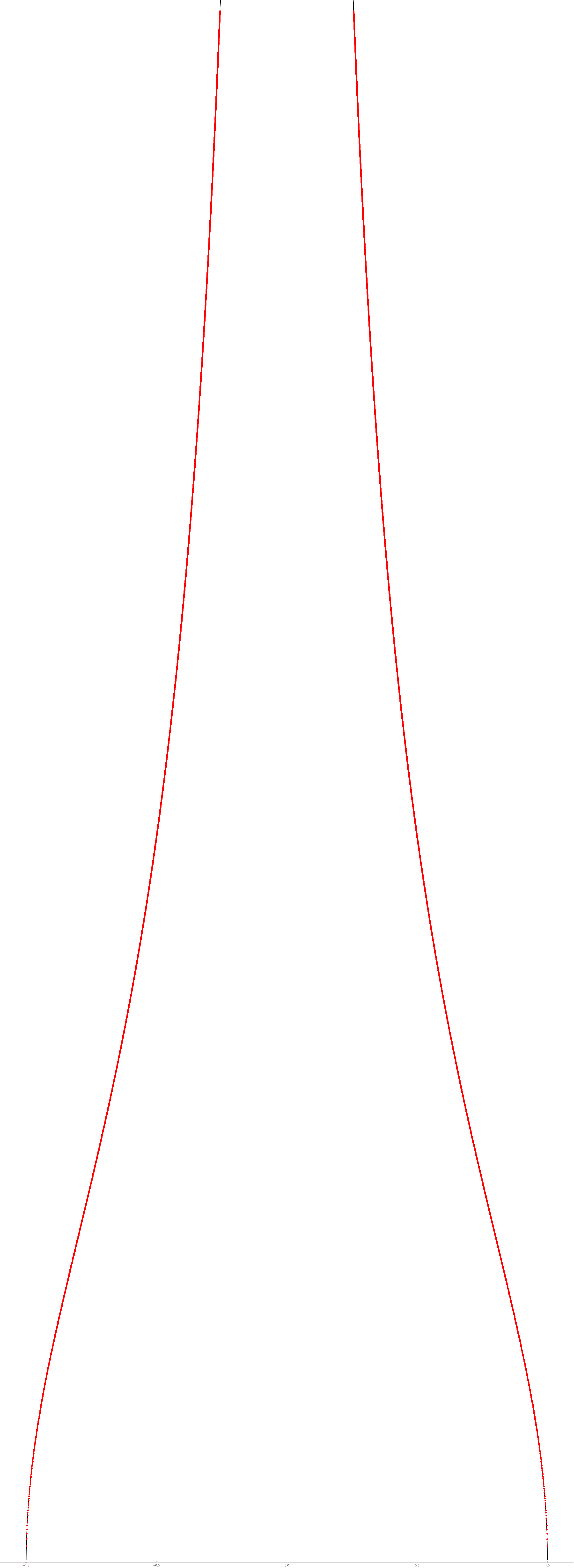}
    \caption{10000 Controlled Oscillations}
    \label{fig:10000controlledoscillations}
\end{minipage}
\end{figure}

\begin{figure}
\begin{minipage}{0.45\textwidth}
    \centering
    \includegraphics[width=.9\textwidth]{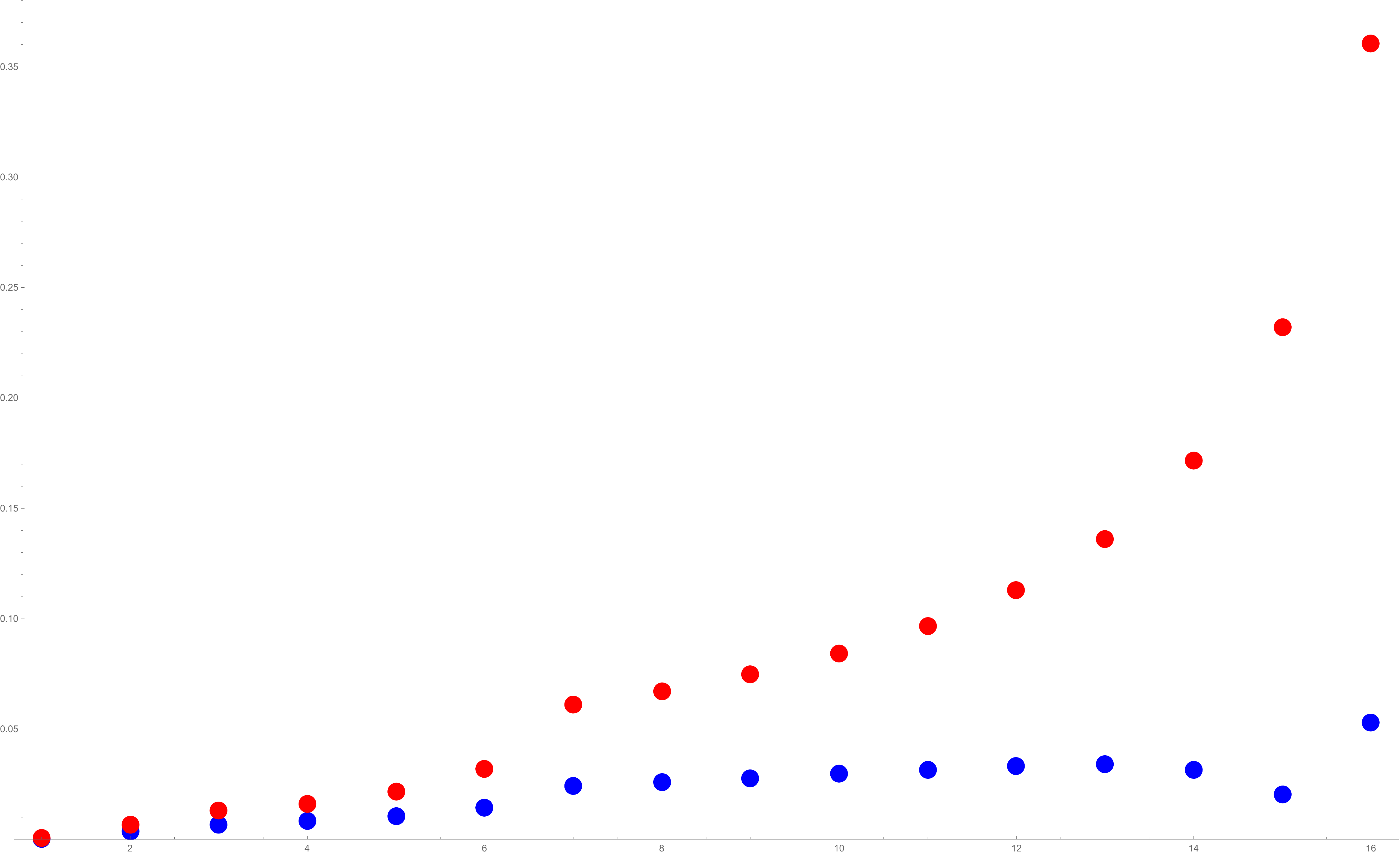}
    \caption{Errors for 1000, 500, 400, 300, 200, 100, 90, 80, 70, 60, 50, 40, 30, 20, 10 Controlled Oscillations}
    \label{fig:errors}
\end{minipage}\hfill
\begin{minipage}{0.45\textwidth}
    \centering
    \includegraphics[width=.9\textwidth]{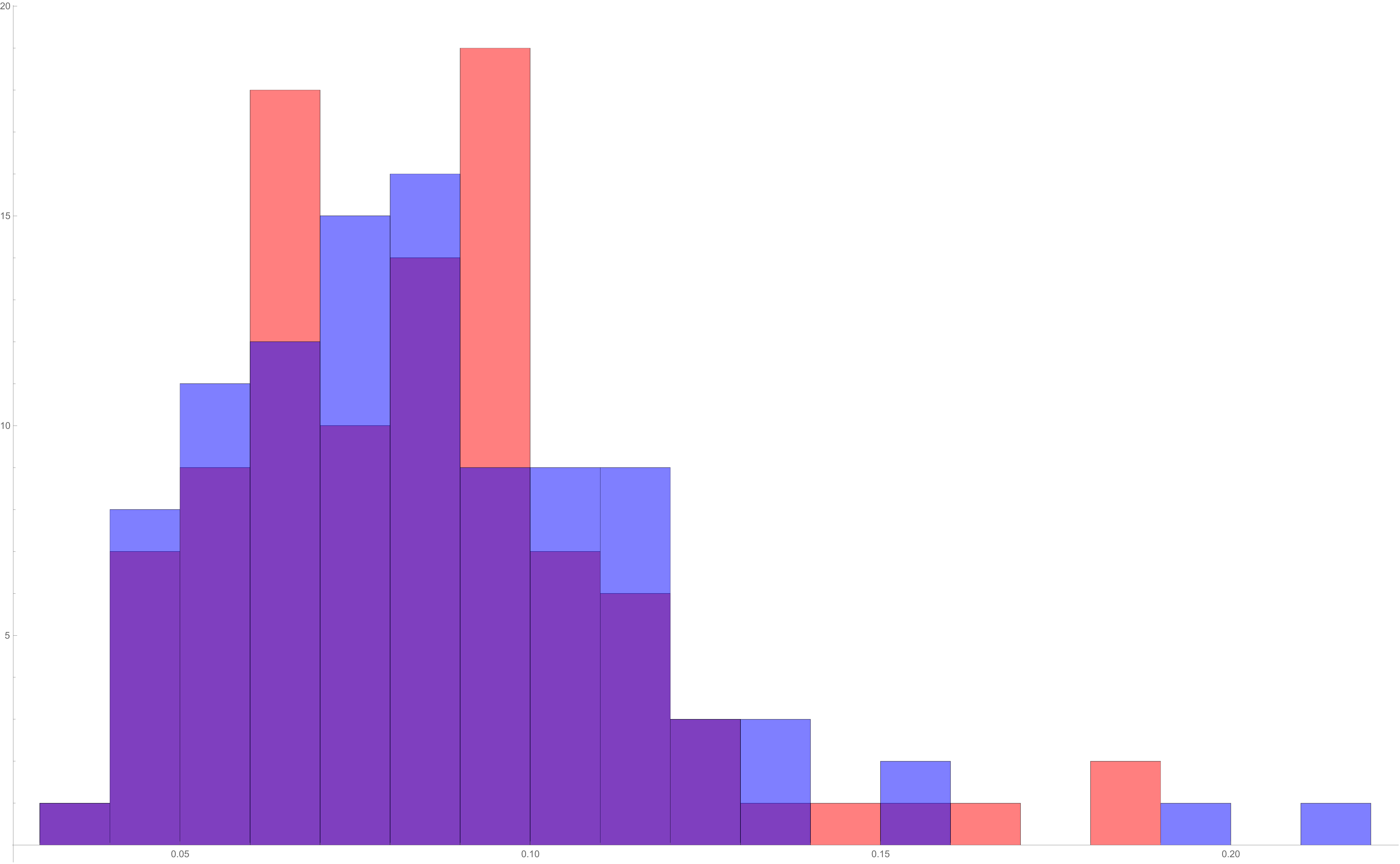}
    \caption{Histogram of 100 Errors for 1000 Random Oscillations}
    \label{fig:100errors}
\end{minipage}
\end{figure}

\begin{figure}
\begin{minipage}{0.45\textwidth}
    \centering
    \includegraphics[width=0.9\textwidth,height=2in]{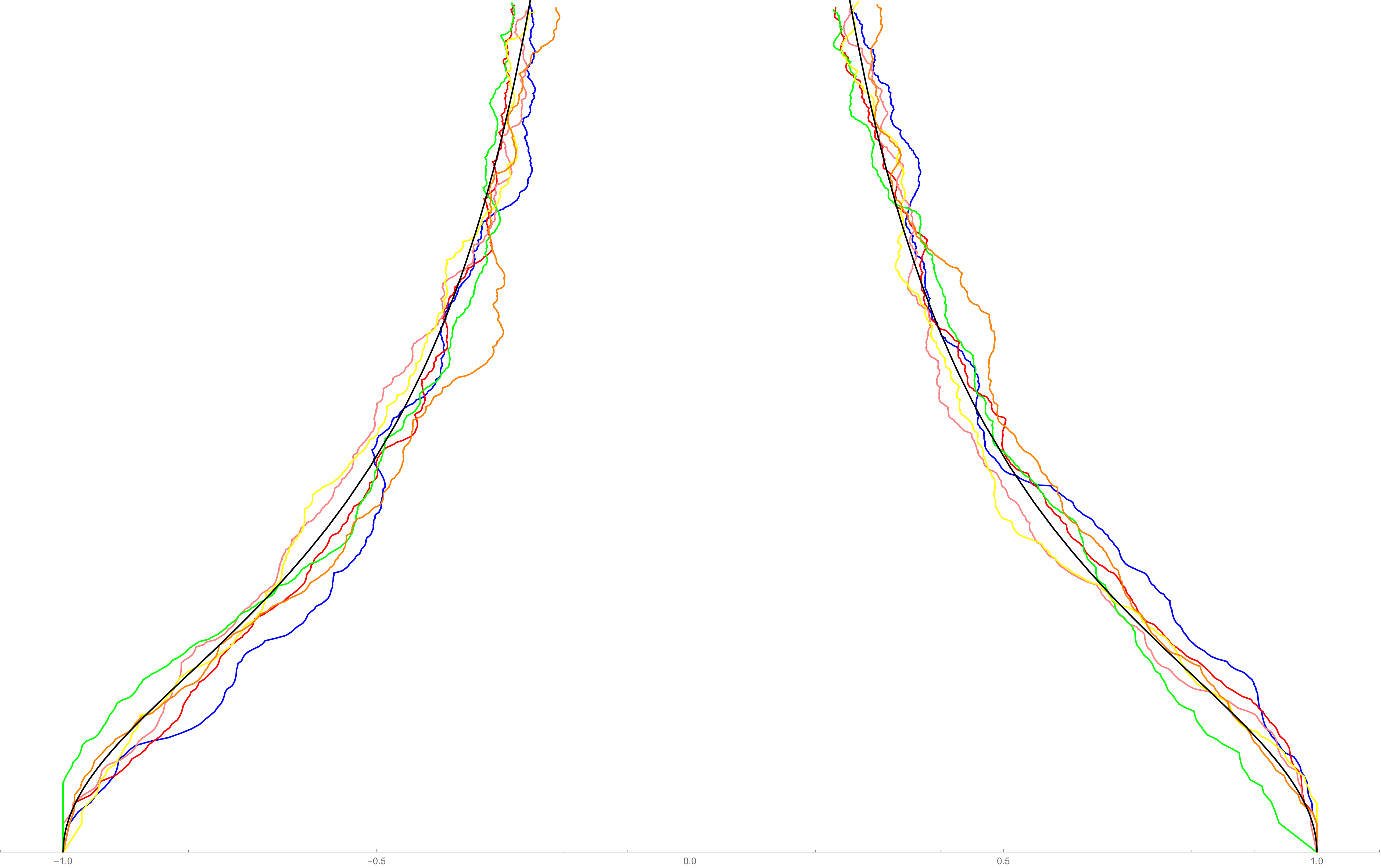}
    \caption{1000 Random Oscillations}
    \label{fig:1000randomoscillations}
\end{minipage}\hfill
\begin{minipage}{0.45\textwidth}
    \centering
    \includegraphics[width=0.9\textwidth,height=2in]{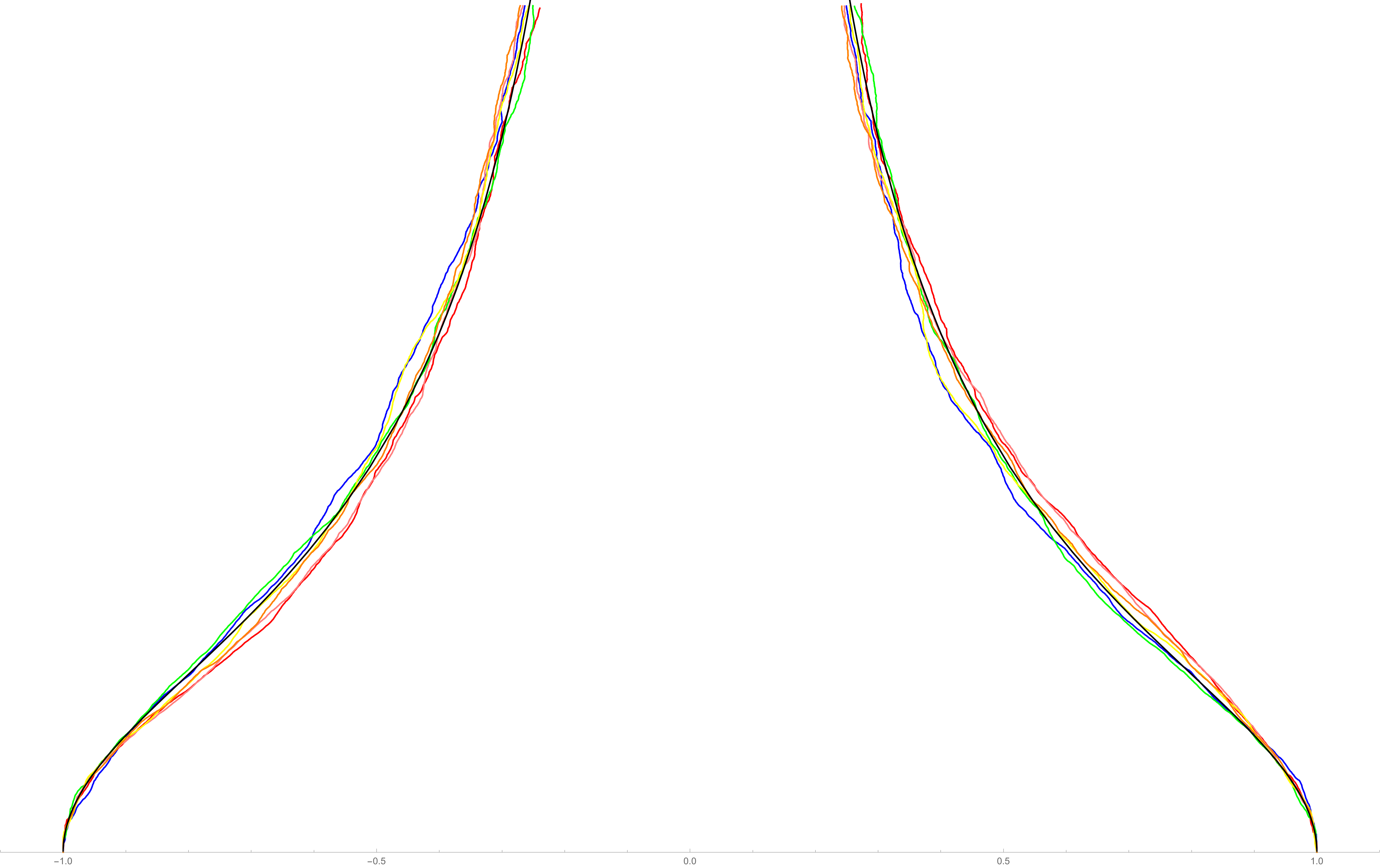}
    \caption{10000 Random Oscillations}
    \label{fig:10000randomoscillations}
\end{minipage}
\end{figure}

\section{Background}\label{background}
We now give a more rigorous introduction to the Loewner equation, hulls, and prime ends. This section gives us the tools and background needed to generalize Theorem 1.1 in \cite{RS14} which we used to prove Proposition \ref{prop:nonconstantweights}. We begin by reintroducing the Loewner equation. Next we discuss hulls in the upper half-plane. This leads to the section on Loewner hulls, which are hulls that can be generated through the Loewner equation driven by a continuous driving function. We then generalize the notion of the tip of a curve to prime ends. This section concludes with results on multiple Loewner hulls, which are hulls that can be generated through the multiple Loewner equation driven by multiple continuous driving functions.

\subsection{Loewner Equation}\label{loewnerequation}

Let $\lambda:[0,T]\to\mathbb{R}$ be continuous. For $z\in\mathbb{H}$, the (single, chordal) Loewner equation is the initial value problem
\begin{equation}\label{eqn:LE}
    \frac{\partial}{\partial t} g_t(z)
    =\frac{2}{g_t(z)-\lambda(t)},
    \quad g_0(z)=z.
\end{equation}
A solution to the Loewner equation exists on some time interval, where the only issue stopping existence is when $g_t(z)=\lambda(t)$. We denote $K_t$ as the points of $\mathbb{H}$ when the solution has failed to exist at some time up to time $t$, that is,
\begin{equation}
    K_t=\{z\in\mathbb{H}:g_s(z)=\lambda(s)\text{ for some }s\in[0,t]\}.
\end{equation}
The function $\lambda$ is called the driving function and $(g_t)_{t\in[0,T]}$ is called a Loewner chain. For $t\in[0,T]$, we call $K_t$ a Loewner hull and we call the family $(K_t)_{t\in[0,T]}$ a Loewner family (see Section \ref{loewnerhulls}). We introduce the Loewner hull moniker to distinguish hulls that can be generated by a single, continuous driving function from hulls that cannot. For example, using $\lambda(t)=c$, we can grow a vertical line starting at $c$. However, two vertical lines at $c_1$ and $c_2$ (with $c_1\not=c_2$) cannot be generated from a single continuous driving function. We discuss this further in Section \ref{loewnerhulls}. The solution $g_t(z)$ is the conformal map from $\mathbb{H}\setminus K_t$ onto $\mathbb{H}$ that satisfies
\begin{equation}
    g_t(z)=z+\frac{2t}{z}+O\left(\frac{1}{z^2}\right)
\end{equation}
near infinity. We define the half-plane capacity of $K_t$, $\hcap(K_t)$, to be $2t$ (see Section \ref{hulls}).

If instead of starting with a continuous function, we started with a Loewner family, we can find a unique driving function satisfying (\ref{eqn:LE}). This gives a one-to-one correspondence between continuous functions and Loewner families of hulls. See \cite{lawler} Lemma 4.2, Theorem 4.6, and the discussion following Example 4.12 for more details.

Now, let $\lambda_1,...,\lambda_n:[0,T]\to\mathbb{R}$ be continuous and $w_1,...,w_n\in L^1[0,T]$ with $\sum_{k=1}^{\infty}w_k(t)\equiv 1$. For $z\in\mathbb{H}$, the multiple Loewner equation is the initial value problem
\begin{equation}\label{eqn:multiLE}
    \frac{\partial}{\partial t} g_t(z)=\sum_{k=1}^{n}\frac{2w_k(t)}{g_t(z)-\lambda_k(t)}\text{ a.e. }t\in [0,T],
    \quad g_0(z)=z.
\end{equation}
This is the sum of weighted Loewner equations, which allows growth of multiple Loewner hulls simultaneously.
Note that (\ref{eqn:multiLE}) holds a.e. $t\in [0,T]$ whereas (\ref{eqn:LE}) holds for all $t\in[0,T]$.

\subsection{Hulls}\label{hulls}

\begin{defn}
A bounded set $K\subseteq\mathbb{H}$ is a hull if $\mathbb{H}\setminus K$ is simply connected.
\end{defn}

\noindent For any hull $K$, there is a unique conformal map $g_K:\mathbb{H}\setminus K\to\mathbb{H}$ with $\lim_{z\to\infty}(g_K(z)-z)=0$, by Riemann mapping theorem (see Proposition 3.36 in \cite{lawler}). The inverse of $g_K$ satisfies the Nevanlinna representation formula
\begin{equation}
    g_K^{-1}(z)=z+\int_{\mathbb{R}}\frac{d\mu_K(t)}{t-z}
\end{equation}
for some finite, nonnegative Borel measure on $\mathbb{R}$ (see Section 3.1 in \cite{schleissinger}).
We now state a very useful result from \cite{RS14}.

\begin{lemma}[3.4 \cite{RS14}]\label{lem:3.2.3}
Let $A$ be a hull.
\begin{itemize}
    \item[(a)] If $\overline{A}\cap\mathbb{R}$ is contained in the closed interval $[a,b]$, then $g_A(\alpha)\leq\alpha$ for every $\alpha\in\mathbb{R}$ with $\alpha<a$ and $g_A(\beta)\geq\beta$ for every $b\in\mathbb{R}$ with $\beta>b$.
    \item[(b)] If the open interval $(a,b)$ is contained in $\mathbb{R}\setminus\overline{A}$, then $|g_A(\beta)-g_A(\alpha)|\leq|\beta-\alpha|$ for all $\alpha,\beta\in(a,b)$.
\end{itemize}
\end{lemma}

\begin{defn}
Let $K$ be a hull. The half-plane capacity of $K$ is defined as
\begin{equation}
    \hcap(K)=\lim_{z\to\infty}z(g_K(z)-z).
\end{equation}
\end{defn}
\noindent Half-plane capacity is a real value relating $g_K$ and $K$. Part of the importance of the half-plane capacity is captured in the following lemma from \cite{RS14}.

\begin{lemma}[3.1 \cite{RS14}]\label{lem:3.2.2}
Let $A$, $A_1$, $A_2$ be hulls.
\begin{itemize}
    \item[(a)] If $A_1\cup A_2$ and $A_1\cap A_2$ are hulls, then
    \begin{equation}
        \hcap(A_1)+\hcap(A_2)\geq\hcap(A_1\cup A_2)+\hcap(A_1\cap A_2)
    \end{equation}
    \item[(b)] If $A_1\subset A_2$, then $\hcap(A_2)=\hcap(A_1)+\hcap(g_{A_1}(A_2\setminus A_1))\geq\hcap(A_1)$.
    \item[(c)] If $A_1\cup A_2$ is a hull and $A_1\cap A_2=\emptyset$, then $\hcap(g_{A_1}(A_2))\leq\hcap(A_2)$.
    \item[(d)] If $c>0$, then $\hcap(cA)=c^2\hcap(A)$ and $\hcap(A\pm c)=\hcap(A)$.
\end{itemize}
\end{lemma}

Remark 3.50 in \cite{lawler} gives that there exists $M>0$ so that for any hull $K$,
\begin{equation}
    \text{diam}(g_K(K)) < M\text{diam}(K).
\end{equation}
In order to further discuss diam$g_K(K)$, we introduce some notation.

\begin{defn}
Let $A$ and $B$ be hulls or a finite union of hulls. Let $g_B:\mathbb{H}\setminus B\to\mathbb{H}$ be the hydrodynamically normalized conformal map. Define $g_B^+(A)=0$ if $A\subseteq\text{int}(B)$ and otherwise
\begin{equation}
    g_{B}^+(A)
    =\max\left\{\lim_{n\to\infty}g_{B}(z_n):(z_n)_{n=1}^{\infty}\subseteq\mathbb{H}\setminus B, z_n\to z\in A,g_{B}(z_n)\to x\in\mathbb{R}\right\}.
\end{equation}
Similarly, define $g_B^-(A)=0$ if $A\subseteq\text{int}(B)$ and otherwise
\begin{equation}
    g_{B}^-(A)
    =\min\left\{\lim_{n\to\infty}g_{B}(z_n):(z_n)_{n=1}^{\infty}\subseteq\mathbb{H}\setminus B, z_n\to z\in A,g_{B}(z_n)\to x\in\mathbb{R}\right\}.
\end{equation}
\end{defn}

This means
\begin{equation}\label{eqn:lawlerremark3.50}
    g_K^+(K)-g_K^-(K)=\text{diam}(g_K(K))\leq M\text{diam}(K)
\end{equation}

\subsection{Loewner Hulls}\label{loewnerhulls}

As previously mentioned, not all hulls can be grown from the Loewner equation driven by a continuous function, for instance a tree or a disconnected set. We will call these special hulls Loewner hulls.

\begin{defn}
We say that a family of hulls, $(K_t)_{t\in[0,T]}$ is a Loewner family if for all $t\in[0,T]$, $\hcap(K_t)=2t$, $K_s\subset K_t$ for $s<t$, and for all $\epsilon>0$ there exists $\delta>0$ so that for $t\in [0,T-\delta]$ there is a bounded, connected set $S\subset\mathbb{H}\setminus K_t$ with diam$(S)<\epsilon$ where $S$ disconnects $K_{t+\delta}\setminus K_t$ from infinity in $\mathbb{H}\setminus K_t$.
\end{defn}

\noindent The above definition is motivated by Theorem 2.6 of \cite{lsw} which states that $(K_t)_{t\in[0,T]}$ is a Loewner family if and only if there exists $\lambda:[0,T]\to\mathbb{R}$ continuous so that $(K_t)_{t\in[0,T]}$ is driven by $\lambda$. Furthermore, $\lambda(t)$ is the point in $\bigcap_{\epsilon>0}g_t(K_{t+\epsilon}\setminus K_t)$. We will say that two Loewner families $(K_t)_{t\in[0,T]}$ and $(L_s)_{s\in[0,S]}$ are disjoint if $\overline{K_T}\cap \overline{L_S}=\emptyset$, where the closure is taken in $\overline{\mathbb{H}}$. Similarly, if $A$ and $B$ are hulls, we say they are disjoint if $\overline{A}\cap \overline{B}=\emptyset$. When there is no risk of confusion, we denote Loewner families simply by $K_t$, dropping the index on $t$.

\begin{defn}
We say that the hull $K$ with $\hcap(K)=2T$ is a Loewner hull if there is a Loewner family $K_t$ with $K_T=K$.
\end{defn}

\noindent The relationship between a Loewner family and its driving function is very deep. We exemplify this relationship by stating a few results that will prove useful.

\begin{lemma}[3.3 (a) \cite{chenrohde}]\label{lem:cr3.3a}
Let $K_t$ be a Loewner family driven by $\lambda$. If $\lambda(t)\in [a,b]$ for all $t\in[0,T]$, then $\overline{K_T}\subset[a,b]\times\mathbb{R}$.
\end{lemma}

\begin{lemma}[4.13 \cite{lawler}]\label{lem:lawler4.13}
Let $K_t$ be a Loewner family generated by $\lambda$ with Loewner chain $g_t$. Define $R_t=\max\{\sqrt{t},\sup\{|\lambda(s)|:0\leq s\leq t\}\}$. Then $\sup\{|z|:z\in K_t\}\leq 4R_t$. In fact, if $|z|>4R_t$, then $|g_s(z)-z|\leq R_t$ for $0\leq s\leq t$.
\end{lemma}

\noindent Beyond the driving function, Loewner families can only grow in particular ways.

\begin{defn}[\cite{lawler}]
Let $K_t$ be a Loewner family. We call $z$ a $t$-accessible point if $z\in K_t\setminus\cup_{s<t} K_s$ and there exists a continuous curve $\gamma:[0,1]\to\mathbb{C}$ with $\gamma(0)=z$ and $\gamma(0,1]\subseteq\mathbb{H}\setminus K_t$.
\end{defn}

\begin{prop}[4.26 \cite{lawler}]\label{pro:lawler4.26}
If $t>0$ and $z$ is a $t$-accessible point, then there is a strictly increasing sequence $s_j\uparrow t$ and a sequence of $s_j$-accessible points $z_j$ with $z_j\to z$.
\end{prop}

\begin{prop}[4.27 \cite{lawler}]\label{pro:lawler4.27}
For each $t>0$, there is at most one $t$-accessible point. Also, the boundary of the time $t$ hull is contained in the closure of the set of $s$-accessible points for $s\leq t$.
\end{prop}

\noindent The restriction on the number of $t$-accessible points also shows that the boundary of a hull always intersects the boundary of previous hulls.

\begin{lemma}\label{lem:B}
Let $K_t$ be a Loewner family generated by $\lambda$. Fix $0<t\leq T$. Then there exists $0<s<t$ so that $\partial_{\mathbb{H}} K_t\cap K_s\not=\emptyset$. Moreover, $\partial_{\mathbb{H}} K_t\cap \partial K_r\not=\emptyset$ for $s\leq r\leq t$.
\end{lemma}

\noindent Note that here we use $\partial_{\mathbb{H}}$ to indicate the boundary with respect to $\mathbb{H}$. Explicitly, for $A\subseteq\mathbb{H}$,
\begin{equation}
    \partial_{\mathbb{H}} A=\{z\in\mathbb{C}:\text{ exists }(z_n)_{n=1}^\infty\subseteq\mathbb{H}\setminus A\text{ with }z_n\to z\}
\end{equation}

\begin{proof}
Suppose not - that is, for some fixed $t\in (0,T]$, $\partial_{\mathbb{H}} K_t\cap K_s=\emptyset$ for all $0<s<t$. Since $0<t$, we have that $\partial_{\mathbb{H}} K_t$ is larger than a singleton set. Let $z_1,z_2\in\partial_{\mathbb{H}} K_t$ with $|z_1-z_2|=\delta>0$. Then there are $w_1,w_2\in\mathbb{H}\setminus K_t$ with $|z_i-w_i|<\frac{\delta}{3}$ for $i=1,2$. Let $\gamma_i:[0,1]\to\mathbb{H}$ be the straight line segment starting at $w_i$ and ending at $z_i$ for $i=1,2$. Let $t_i\in(0,1]$ be the first time that $\gamma_i$ intersects $K_t$ and $z_i'=\gamma_i(t_i)$. Two important facts follow. First, since $z_i'\in\partial_{\mathbb{H}} K_t\subseteq K_t\setminus\bigcup_{s<t}K_s$ for $i=1,2$, $z_1'$ and $z_2'$ are $t$-accessible. Second, by construction $|z_1'-z_2'|>\frac{\delta}{3}$, so $z_1'\not= z_2'$. This shows that there is more than one $t$-accessible point, a contradiction to Proposition \ref{pro:lawler4.27}. So, for all $t\in(0,T]$ there is $0<s<t$ with $\partial_{\mathbb{H}} K_t\cap K_s\not=\emptyset$.

The moreover statement follows immediately using the fact that $s\leq r\leq t$ gives $K_s\subseteq K_r\subseteq K_t$.
\end{proof}

\noindent Often we will be considering the family $(g_L(K_t))_{t\in [0,T]}$ where $L$ is a hull disjoint from $K_T$. The next lemma investigates what happens when a Loewner family is conformally transformed.

\begin{lemma}[2.8 \cite{lsw}]\label{lem:lsw2.8}
Let $(K_t)_{t\in[0,T]}$ be a Loewner family driven by $\lambda$. Let $D$ be a relatively open subset of $\overline{\mathbb{H}}$ which contains $\overline{K_T}$, and set $D_{\mathbb{R}}:=D\cap\mathbb{R}$. Let $G:D\to\overline{\mathbb{H}}$ be conformal in $D\setminus D_{\mathbb{R}}$ and continuous in $D$, and suppose that $G(D_{\mathbb{R}})\subset\mathbb{R}$. Then $(G(K_t))_{t\in[0,T]}$ is a Loewner family. Moreover, 
$\partial_t [\hcap(G(K_t))]=G'(\lambda(0))^2\partial_t \hcap(K_t)$ as $t=0$.
\end{lemma}

\subsection{Prime Ends}

In order to generalize the results of \cite{RS14}, we need to generalize the tip of a curve into the setting of hulls. This is done with prime ends, which are equivalence classes of crosscuts. We give only a brief introduction, for more details see \cite{rempe-gillen}.

\begin{defn}[\cite{rempe-gillen}] Let $\Omega\subseteq\mathbb{H}$ be a simply connected domain containing $\infty$. Let $C$ be a crosscut of $\Omega$ (that is, a Jordan arc in $\Omega$ with endpoints in $\partial\Omega$) and $\Omega_C$ the component of $\Omega\setminus C$ not containing $\infty$. A prime end of $\Omega$ is represented by a sequence of pairwise disjoint crosscuts $(C_n)_{n=1}^{\infty}$ with $\text{diam}(C_n)\to 0$ as $n\to\infty$ and $C_{n+1}\subseteq\overline{\Omega_{C_n}}$. Two sequences, $(C_n)_{n=1}^{\infty}$ and $(\widetilde{C}_n)_{n=1}^{\infty}$, represent the same prime end if for each $n$ there is a $J_n\in\mathbb{N}$ so that $\widetilde{C}_j\subseteq\Omega_{C_n}$ for $j\geq J_n$ and vice versa.
\end{defn}

\begin{defn}
Let $p$ be a prime end represented by the sequence of crosscuts $(C_n)_{n=1}^{\infty}$. The impression of $p$ is defined as $I(p)=\bigcap_{n=1}^{\infty}\overline{\Omega_{C_n}}.$
Since $(\overline{\Omega_{C_n}})_{n=1}^{\infty}$ is a decreasing sequence of nonempty, compact, and connected sets, the impression of $p$ is nonempty. Moreover, the impression of $p$ is independent of its representation.
\end{defn}

\begin{lemma}\label{lem:C}
Let $K_t$ be a Loewner family generated by $\lambda$. Fix $0<t\leq T$. If there exists $0<s<t$ such that $\lambda(s)<\lambda(r)$ or $\lambda(s)>\lambda(r)$ for $r\in(s,t)$, then $\overline{K_s}\cap\partial_{\mathbb{H}} K_t\not=\emptyset$.
\end{lemma}

\begin{proof}
Suppose $\lambda(s)<\lambda(r)$ (resp. $\lambda(s)>\lambda(r)$) for $s<r<t$. Then Lemma \ref{lem:cr3.3a} shows that $\lambda(s)\leq\min\{\overline{g_{K_s}(K_t\setminus K_s)}\cap\mathbb{R}\}$ ($\geq\max$ resp.). As $\lambda(s)\in\overline{g_{K_s}(K_t\setminus K_s)}$, $\lambda(s)\in\partial\overline{g_{K_s}(K_t\setminus K_s)}$. Now, there exists $(w_n)_{n=1}^{\infty}\subset\mathbb{H}\setminus\overline{g_{K_s}(K_t\setminus K_s)}$ with $w_n\to\lambda(s)$. So, there exists a corresponding sequence $(z_n)_{n=1}^{\infty}\subset\mathbb{H}\setminus K_t$ so that $g_{K_s}(z_n)=w_n$. Furthermore, there is a subsequence of $(z_n)_{n=1}^{\infty}$ that converges to a point in $\overline{K_s}$ as there is at least one point in the impression of the prime end corresponding to $\lambda(s)$. This shows that $\overline{K_s}\cap\partial_{\mathbb{H}} K_t\not=\emptyset$.
\end{proof}

\begin{defn}
Let $\Omega\subseteq\mathbb{H}$ be a simply connected domain containing $\infty$. Let $P(\Omega)$ denote the set of prime ends of $\Omega$ and $\widehat{\Omega}:=\Omega\cup P(\Omega)$ denote the Carath\'eodory compactification of $\Omega$.
We can define a topology on $\widehat{\Omega}$ by making the following equivalent:

\begin{itemize}
    \item $(z_j)_{j=1}^{\infty}\subseteq\Omega$ converges to $p\in P(\Omega)$
    \item for any $(C_n)_{n=1}^{\infty}\in p\in P(\Omega)$ there exists $J\in\mathbb{N}$ so that $(z_j)_{j=J}^{\infty}\subseteq\Omega_{C_n}$
\end{itemize}
\end{defn}

\noindent Under this topology, if $g:\Omega\to\mathbb{H}$ is conformal, then $g$ extends to a homeomorphism $\widehat{g}:\widehat{\Omega}\to\overline{\mathbb{H}}$. We can identify prime ends of $\Omega$ with boundary points of $\Omega$ as follows:

\begin{equation}
 (z_j)_{j=1}^{\infty}\subseteq\Omega
 \text{ with }
 z_j\to z\in \partial\Omega
 \text{ if and only if }
 (z_j)_{j=1}^{\infty}\subseteq\Omega
 \text{ with }
 z_j\to p\in P(\Omega)
\end{equation}

\noindent If $z\in\partial\Omega$ and $p\in P(\Omega)$ are identified, we do not distinguish the point $z$ and the prime end $p$.

\noindent Since the identity map on $\mathbb{H}$ is conformal, $\overline{H}$ and $\widehat{H}$ are homeomorphic and we can think of boundary points (i.e. real points) as prime ends and the other way around.

\begin{defn}
Let $K_t$ be a Loewner family driven by $\lambda$ with Loewner chain $g_t$. Let $p$ be a prime end of $\mathbb{H}\setminus K_t$.
We say that ``$p$ corresponds to $\lambda(t)$'' or ``$p$ is the (generalized) tip of $K_t$'' if $\widehat{g}_t(p)=\lambda(t)$.
\end{defn}

\noindent This gives us a family of prime ends $(p_t)_{t\in[0,T]}$ each corresponding to $\lambda(t)$ which generates $K_t$. More specifically, $\widehat{g}_t(p_t)=\lambda(t)$ where $g_t$ is the Loewner chain corresponding to $K_t$ and $\lambda$ is its driving function.

In the situation of a curve $\gamma$ with Loewner chain $g_t$, since $g_t(\gamma(t))=\lambda(t)$, the tip at time $t$, $\gamma(t)$, is the prime end corresponding to $\lambda(t)$. This is the reason that we use prime ends to generalize tips.

We now will revisit the definitions of $g_B^+(A)$ and $g_B^-(A)$ and relate them to prime ends. If $A\not\subseteq\text{int}(B)$,
\begin{equation}
    g_B^+(A)=\sup\{g_B(p)\in\mathbb{R}:p\in P(\mathbb{H}\setminus A),I(p)\cap\overline{A}\not=\emptyset\}
\end{equation}
and
\begin{equation}
    g_B^-(A)=\inf\{g_B(p)\in\mathbb{R}:p\in P(\mathbb{H}\setminus A),I(p)\cap\overline{A}\not=\emptyset\}.
\end{equation}
This follows from $g_B$ extending to $\widehat{\mathbb{H}\setminus B}$. Note that from now on, we will assume $g_B$ is its extension $\widehat{g}_B$.

\subsection{Multiple Loewner Hulls}\label{multiloewnerhulls}

We now switch to the setting of our main result: multiple, disjoint Loewner families. Let $K$ and $L$ be disjoint hulls. There are many ways that $K\cup L$ can be mapped down to the real line. Two basic ways are mapping down one hull and then mapping down the image other hull, see Figure \ref{fig:maponethenother}. By uniqueness we have
\begin{equation}\label{eqn:mapcomposition}
    g_{g_K(L)}\circ g_K=g_{K\cup L}=g_{g_L(K)}\circ g_L.
\end{equation}
This gives a significant amount of flexibility in our maps.

\begin{figure}
\centering
\begin{tikzpicture}

\draw (0,0) -- (6,0);
\draw[thick] (0.5,0) 
    to [out=90, in=270] (1,0.75)
    to [out=90, in=270] (0.5,1.5)
    to [out=90, in=180] (0.875,1.875)
    to [out=0, in=90] (1.25,1.5)
    to [out=270, in=180] (1.5,1.25)
    to [out=0, in=270] (1.875,2)
    to [out=90, in=90] (2.5,1.25)
    to [out=270, in=90] (2.25,0.5)
    to [out=270, in=135] (2.75,0);
\node[above] at (1.625,0) {$K_t$};
\draw[fill] (0.875,1.875) circle [radius=0.05];
\node[above] at (0.875,1.875) {$p_t$};
\draw[thick] (3.25,0) 
    to [out=90, in=270] (3.5,0.75)
    to [out=90, in=270] (3.25,1.5)
    to [out=90, in=180] (3.75,2)
    to [out=0, in=90] (4.25,1.5)
    to [out=270, in=180] (4.75,1)
    to [out=0, in=90] (5.5,0);
\node[above] at (4.375,0) {$L$};

\draw (9,0) -- (15,0);
\draw[thick] (10.17,0)
    to [out=90.242,in=257.702] (10.6135,0.611557)
    to [out=84.752,in=264.509] (10.2432,1.30248)
    to [out=84.3251,in=172.591] (10.6219,1.60237)
    to [out=-7.88783,in=80.947] (10.9066,1.21776)
    to [out=-99.053,in=169.8] (11.0813,0.968246)
    to [out=-10.2,in=165.7926] (11.5572,1.54393)
    to [out=-14.2074,in=63.8903] (11.8558,0.73348)
    to [out=243.8903,in=77.6203] (11.5178,0.303857)
    to [out=257.6203,in=83] (11.75,0);
\node[below] at (10.96,0) {$g_L(K_t)$};
\draw[fill] (10.6219,1.60237) circle [radius=0.05];
\node[above] at (10.6219,1.60237) {$g_L(p_t)$};

\draw (0,-4) -- (6,-4);
\draw[thick] (4.03045,-4)
    to[out=93.3425,in=-63.638](4.03006,-3.64881)
    to[out=116.362,in=-48.1169](3.55863,-3.32399)
    to[out=131.883,in=202.392](3.73849,-2.64428)
    to[out=382.392,in=106.857](4.33724,-2.92541)
    to[out=286.857,in=189.0628](4.83518,-3.23488)
    to[out=9.063,in=90.2936](5.5,-4);
\node[below] at (4.77,-4) {$g_{K_t}(L)$};

\draw (9,-4) -- (15,-4);

\draw[->] (6.5,1) to [out=0,in=180] (8.5,1);
\node[above] at (7.5,1) {$g_L$};
\draw[->] (6.5,-3) to [out=0,in=180] (8.5,-3);
\node[below] at (7.5,-3) {$g_{g_{K_t}(L)}$};
\draw[->] (3,-0.5) to [out=270,in=90] (3,-1.5);
\node[left] at (3,-1) {$g_{K_t}$};
\draw[->] (12,-0.5) to [out=270,in=90] (12,-1.5);
\node[right] at (12,-1) {$g_{g_L(K_t)}$};
\draw[->] (6.5,-0.5) to [out=333.43,in=153.47] (8.5,-1.5);
\node[above,rotate=333.43] at (7.5,-1) {$g_{K_t\cup L}$};

\draw[fill] (1.5,-4) circle [radius=0.05];
\node[below] at (1.5,-4) {$U(t)$};
\draw[fill] (10.5,-4) circle [radius=0.05];
\node[below] at (10.5,-4) {$\lambda(t)$};

\end{tikzpicture}
\caption{Mapping Down Hulls in Different Orders}
\label{fig:maponethenother}
\end{figure}
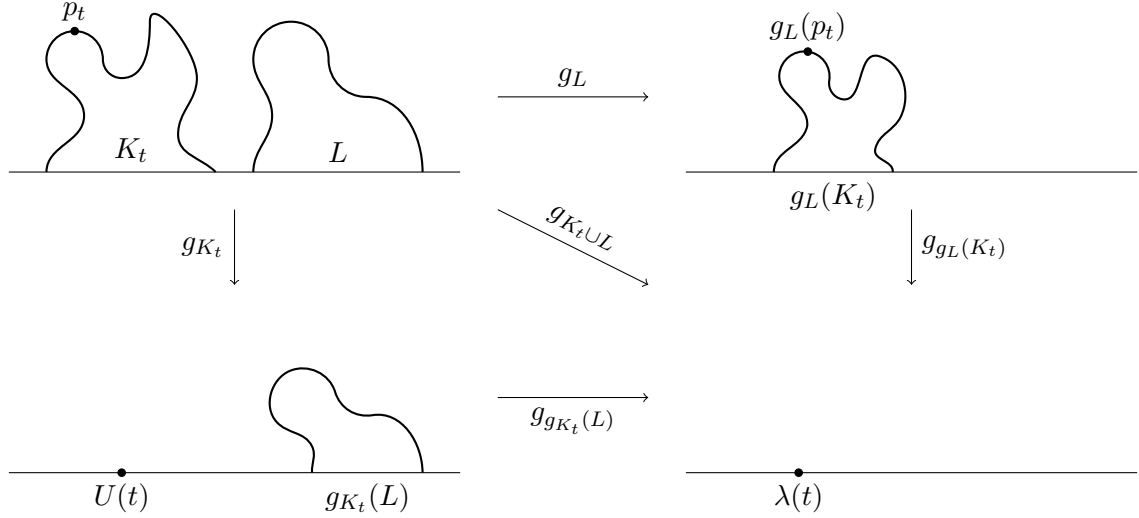

\noindent We now state a few preliminary results on what happens when another hull is added.

\begin{lemma}\label{lem:3.6.8c}
Let $K_t$ be a Loewner family and $L$ a hull disjoint from $K_T$. If $K_s\cap\partial_{\mathbb{H}} K_t\not=\emptyset$, then for $s\leq r\leq t$,
\begin{equation}
    g_{K_t\cup L}^-(K_t\setminus K_s)
    \leq g_{K_t\cup L}^-(K_t\setminus K_r)
    \leq g_{K_t\cup L}^+(K_t\setminus K_r)
    \leq g_{K_t\cup L}^+(K_t\setminus K_s)
\end{equation}
\end{lemma}

\begin{proof}
The middle inequality follows from the definitions of $g_{K_t\cup L}^-$ and $g_{K_t\cup L}^+$.

For the first inequality, let $(z_n)_{n=1}^{\infty}\subseteq\mathbb{H}\setminus(K_t\cup L)$ with $z_n\to z\in K_t\setminus K_r$ and $g_{K_t\cup L}\to x\in\mathbb{R}$. Then as $K_s\subseteq K_r$, $z\in K_t\setminus K_s$. So, $g_{K_t\cup L}(K_t\setminus K_s)\leq x$. This holds for any such sequence, so the first inequality is proven.

The third inequality follows in the same manner.
\end{proof}

Let $K_t$ be a Loewner family driven by $U:[0,T]\to\mathbb{R}$ and $L$ be a hull disjoint from $K_T$. What happens to $U$ if we map down $L$ and then map down $g_L(K_t)$? What happens to $U$ if we do the opposite and map down $K_t$ then $L$? The answer is actually given using (\ref{eqn:mapcomposition}) and $g_{K_t}(p_t)=U(t)$ for the corresponding family of prime ends $p_t$. Observe:
\begin{equation}\label{eqn:timemapcomposition}
    g_{g_{K_t}(L)}(U(t))=g_{g_{K_t}(L)}(g_{K_t}(p_t))=g_{g_{L}(K_t)}(g_L(p_t)).
\end{equation}
If we define $\lambda(t)=g_{g_{K_t}(L)}(U(t))$, then, as $g_L(p_t)$ is the (generalized) tip of $g_L(K_t)$, $\lambda$ drives $g_L(K_t)$. Moreover, by (\ref{eqn:timemapcomposition}), $\lambda(t)=g_{K_t\cup L}(p_t)$ (see Figure \ref{fig:maponethenother}). Since $p_t$ is the (generalized) tip of $K_t$ in the hull $K_t\cup L$, we get the usual relationship between tips and driving functions. This gives us a concrete way of defining the driving function in the multiple hull setting.

\begin{lemma}\label{lem:A}
Let $K_t$ be a Loewner family driven by $U:[0,T]\to\mathbb{R}$. Let $L$ be a hull disjoint from $K_T$. Let $\lambda(t)=g_{g_{K_t}(L)}(U(t))$. Fix $0\leq s<t\leq T$ so that $K_s\cap\partial_{\mathbb{H}} K_t\not=\emptyset$. Then for $s\leq r\leq t$
\begin{equation}
    g_{K_t\cup L}^-(K_t\setminus K_s)\leq\lambda(r)\leq g_{K_t\cup L}^+(K_t\setminus K_s).
\end{equation}
\end{lemma}

\begin{proof}
Let $0\leq s< t\leq T$, $K_s\cap\partial K_t\not=\emptyset$, and $A_r=g_{K_r\cup L}(K_t\setminus K_r)$ for $s\leq r\leq t$. Then $\lambda(r)\in\mathbb{R}\cap\overline{g_{K_r\cup L}(K_t\setminus K_r)}=\mathbb{R}\cap\overline{A_r}$. Since $g_{K_t\cup L}=g_{A_r}\circ g_{K_r\cup L}$, by Lemma \ref{lem:3.2.3}, $\lambda(r)\in\mathbb{R}\cap\overline{g_{K_t\cup L}(K_t\setminus K_r)}$. So, $g_{K_t\cup L}^-(K_t\setminus K_r)\leq\lambda(r)\leq g_{K_t\cup L}^+(K_t\setminus K_r)$ for $s\leq r\leq t$.

Let $s<r<t$. Then as $K_s\subset K_r$ and $K_s\cap\partial_{\mathbb{H}} K_t\not=\emptyset$, we have $K_r\cap\partial_{\mathbb{H}} K_t\not=\emptyset$. Using Lemma \ref{lem:3.6.8c},
\begin{equation}
    g_{K_t\cup L}^-(K_t\setminus K_s)
    \leq g_{K_t\cup L}^-(K_t\setminus K_r)
    \leq \lambda(r)
    \leq g_{K_t\cup L}^+(K_t\setminus K_r)
    \leq g_{K_t\cup L}^+(K_t\setminus K_s)
\end{equation}

Lastly, let $r_n\uparrow t$ with $s\leq r_n$. Then for all $n\in\mathbb{N}$
\begin{equation}
    g_{K_t\cup L}^-(K_t\setminus K_s)
    \leq \lambda(r_n)
    \leq g_{K_t\cup L}^+(K_t\setminus K_s)
\end{equation}
As $\lambda$ is continuous, the result holds for $t$.
\end{proof}

\begin{coro}\label{cor:C}
Let $K_t$ be a Loewner family driven by $U:[0,T]\to\mathbb{R}$. Let $L$ be a hull disjoint from $K_T$. Let $\lambda(t)=g_{g_{K_t}(L)}(U(t))$. If $|\lambda(t)-\lambda(s)|>|\lambda(t)-\lambda(r)|$ for $s<r<t$, then $K_s\cap\partial_{\mathbb{H}} K_t\not=\emptyset$.
\end{coro}

\begin{proof}
Since $L\cap K_T=\emptyset$, $g_L(K_t)$ is a Loewner family and furthermore is driven by $\lambda$. If $|\lambda(t)-\lambda(s)|>|\lambda(t)-\lambda(r)|$ for $s<r<t$, then clearly $\lambda(s)\not=\lambda(r)$ for $s<r<t$. Since $\lambda$ is continuous either $\lambda(s)>\lambda(r)$ for all $s<r<t$ or $\lambda(s)<\lambda(r)$ for all $s<r<t$. By Lemma \ref{lem:C}, $\partial_{\mathbb{H}} g_L(K_t)\cap g_L(K_s)\not=\emptyset$. By the disjointness of $K_T$ and $L$, $\partial_{\mathbb{H}} K_t\cap K_s\not=\emptyset$ as well.
\end{proof}

Whenever we use the families $K_t$ and $L_s$, we will assume that $K_T$ is on the left side of $L_S$. We note that the next lemma is a generalization of Lemma 3.5 from \cite{RS14}. The proof of part (a) uses the key ideas brought up in the corresponding proof in \cite{RS14}, but the proof of part (b) is fundamentally different.

\begin{lemma}\label{lem:3.6.8}
Let $(K_t)_{t\in[0,T]}$ and $(L_v)_{v\in[0,S]}$ be two disjoint Loewner families. Then, for any $t\in[0,T]$ and $s\in[0,S]$,
\begin{itemize}
    \item[(a)] $g_{K_T\cup L_S}^-(K_T)\leq g_{K_t\cup L_s}^-(K_T)< g_{K_t\cup L_s}^+(L_S)\leq g_{K_T\cup L_S}^+(L_S)$
    \item[(b)] $g_{K_t\cup L_s}^-(L_S)-g_{K_t\cup L_s}^+(K_T)\geq g_{K_T\cup L_S}^-(L_S)-g_{K_T\cup L_S}^+(K_T)$.
\end{itemize}
\end{lemma}

\begin{proof}[Proof of (a)]
First, the middle inequality is immediate since $\overline{K_T}\cap \overline{L_S}=\emptyset$.

Second, we will prove the first inequality. Let $t\in[0,T]$ and $s\in[0,S]$. Define 
\begin{equation}
    A_1=\overline{g_{K_t\cup L_s}(K_T\setminus K_t)}
    \text{ and }
    A_2=\overline{g_{K_t\cup L_s}(L_S\setminus L_s)}.
\end{equation}
Then $A_1\cap\mathbb{H}$ and $A_2\cap\mathbb{H}$ are disjoint hulls. Let 
$a=g_{K_t\cup L_s}^-(K_T\setminus K_t)$ and
$b=g_{K_t\cup L_s}^+(L_S\setminus L_s)$.
Since $K_T\setminus K_t\subseteq K_T$, $g_{K_t\cup L_s}^-(K_T)\leq a$.
Define $A=A_1\cup A_2$ which is a hull with $\overline{A}\cap\mathbb{R}\subseteq[a,b]$.

If $g_{K_t\cup L_s}^-(K_T)<a$, then by Lemma \ref{lem:3.2.3} (a),
\begin{equation}
    g_{K_T\cup L_S}^-(K_T)=g_A(g_{K_t\cup L_s}^-(K_T))\leq g_{K_t\cup L_s}^-(K_T).
\end{equation}
If $g_{K_t\cup L_s}^-(K_T)=a$, then as $g_A\circ g_{K_t\cup L_s}=g_{K_T\cup L_S}$,
\begin{equation}
    g_{K_T\cup L_S}^-(K_T)= g_A^-(g_{K_t\cup L_s}(K_T))\leq g_{K_t\cup L_s}^-(K_T).
\end{equation}
In both cases, $g_{K_T\cup L_S}^-(K_T)\leq g_{K_t\cup L_s}^-(K_T)$.

Lastly, the other inequality follows in the same manner.
\end{proof}

\begin{proof}[Proof of (b)]
Let $A=\overline{g_{K_t\cup L_s}((K_T\setminus K_t)\cup (L_S\setminus L_s))}$. Then $A\cap\mathbb{H}$ is a hull with
\begin{equation}
    A\cap\mathbb{R}=[g_{K_t\cup L_s}^-(K_T\setminus K_t),g_{K_t\cup L_s}^+(K_T\setminus K_t)]\cup[g_{K_t\cup L_s}^-(L_S\setminus L_s),g_{K_t\cup L_s}^+(L_S\setminus L_s)]
\end{equation}

Let $(x_n)_{n=1}^{\infty},(y_n)_{n=1}^{\infty}\subset\mathbb{R}$ so that $x_n\downarrow g_{K_t\cup L_s}^+(K_T)$, $y_n\uparrow g_{K_t\cup L_s}^-(L_S)$, and
\begin{equation}
    g_{K_t\cup L_s}^+(K_T)<x_n<\frac{g_{K_t\cup L_s}^+(K_T)+g_{K_t\cup L_s}^-(L_S)}{2}<y_n<g_{K_t\cup L_s}^-(L_S)
\end{equation}
Then for every $n$, $0<g_A(y_n)-g_A(x_n)\leq y_n-x_n$ by Lemma \ref{lem:3.2.3} (b) as $(x_n,y_n)\subseteq\mathbb{R}\setminus A$. Since $g_A\circ g_{K_t\cup L_s}=g_{K_T\cup L_s}$,
\begin{equation}
    g_{K_t\cup L_s}^-(L_S)-g_{K_t\cup L_s}^+(K_T)
    \geq g_A^-(g_{K_t\cup L_s}(L_S))-g_A^+(g_{K_t\cup L_s}(K_T))
    =g_{K_T\cup L_S}^-(L_S)-g_{K_T\cup L_S}^+(K_T).
\end{equation}
\end{proof}

\noindent We will now generalize the notion of Loewner families to the multiple hull setting.

\begin{defn}
Let $K_1,...,K_n$ be disjoint Loewner hulls and $\hcap(K_1\cup\cdots\cup K_n)=2T$. For $j=1,...,n$ let $K_t^j$ be an increasing family of hulls so that
\begin{itemize}
    \item $t\mapsto\hcap(K_t^j)$ is nondecreasing
    \item $\hcap(K_t^1\cup\cdots\cup K_t^n)=2t$ for $t\in[0,T]$
    \item $K_T^j=K_j$
\end{itemize}
We call $K_t=(K_t^1,...,K_t^n)$ a Loewner parameterization for the hull $K_1\cup\cdots\cup K_n$.
\end{defn}

\section{Loewner Parameterization Precompactness}\label{precompactness}

The generalization of Theorem 1.1 in \cite{RS14}, Theorem \ref{thm:3.6.2} here, follows with almost the same proof due to prime ends generalizing tips so appropriately. In \cite{RS14} a few technical lemmas are shown, then Theorems 1.1 and 2.2 are proven. Since credit for the proofs goes to the authors of \cite{RS14}, we will state results where the proofs generalize quickly without proof and direct the reader to \cite{RS14}.

\begin{lemma}[3.2 \cite{RS14}]\label{lem:3.6.6}
Let $K_t$ be a Loewner family. Let $L$ be a hull disjoint from $K_T$. Then there exists a constant $c>0$ so that for all $0\leq s<t\leq T$
\begin{equation}
    c\leq\frac{\hcap(K_t\cup L)-\hcap(K_s\cup L)}{t-s}
\end{equation}
\end{lemma}

\begin{lemma}[3.3 \cite{RS14}]\label{lem:3.6.7}
Let $(K_t)_{t\in[0,T_1]}$ and $(L_t)_{t\in[0,T_2]}$ be two disjoint Loewner families. Then there is a constant $c>0$ so that
\begin{equation}
    c\leq\frac{\hcap(K_{t_1}\cup L_{t_2})-\hcap(K_{s_1}\cup L_{s_2})}{t_j-s_j}
\end{equation}
for all $0\leq s_j<t_j\leq T_j$ and $j=1,2$.
\end{lemma}

\begin{lemma}[3.6 \cite{RS14}]\label{lem:3.6.9}
Let $(K_t)_{t\in[0,T]}$ and $(L_v)_{v\in[0,S]}$ be two disjoint Loewner families. Then there exists a constant $M>0$ so that
\begin{equation}
    |g_{K_t\cup L_u}(p)-g_{K_t\cup L_v}(p)|\leq M|v-u|
\end{equation}
for any $t\in[0,T]$ and $u,v\in[0,S]$ where $p$ is the prime end corresponding to $K_t$.
\end{lemma}

The proof of Lemma \ref{lem:3.6.9} from \cite{RS14}, deals with images of base points of slits (specifically, $p_1$ and $p_2$). In particular, the proof looks at the real points that correspond to the prime ends $p_1$ and $p_2$. This is equivalent to mapping down both slits and looking at the corresponding line segments. In order to prove this lemma, we replace $p_1$ by $K_T$ and $p_2$ be $L_S$, which gives the analogue of mapping down both slits. The change from base points of a slit to entire hulls in the proof of Lemma \ref{lem:3.6.9} comes from the fact that for a slit, the two images of the base are the smallest and largest real points in the image of the mapped down slit, whereas with hulls, this corresponds to mapping down the entire hull.

\begin{lemma}[3.7 \cite{RS14}]\label{lem:3.6.10}
Let $K_t$ be a Loewner family driven by $U:[0,T]\to\mathbb{R}$. Let $L$ be a hull disjoint from $K_T$. Let $\lambda(t)=g_{g_{K_t}(L)}(U(t))$. Then there exists $\omega:[0,T]\to[0,\infty)$ increasing with $\lim_{\delta\downarrow 0}\omega(\delta)=\omega(0)=0$ such that
\begin{equation}\label{eqn:3610.3}
    |g_{K_t\cup L}(p_t)-g_{K_s\cup L}(p_s)|\leq\omega(|t-s|)
\end{equation}
for $s,t\in[0,T]$, where $p_t$ and $p_s$ are the prime ends corresponding to $\lambda(t)$ and $\lambda(s)$ respectively.
\end{lemma}

The proof of (\ref{eqn:3610.3}) in the setting of hulls requires more background work than in the setting of slits. The majority of the results in Section \ref{multiloewnerhulls} are used to show that hulls grow similarly to slits. It is this subtle difference in growth that requires a different proof of (\ref{eqn:3610.3}) than in \cite{RS14}. However, the proof that $\omega(\delta)\to0$ as $\delta\to 0$ is the exact same as in \cite{RS14}, so we refer the reader there for the proof.

\begin{proof}
Let $\omega:[0,T]\to[0,\infty)$ be defined by $\omega(0)=0$ and
\begin{equation}
    \omega(\delta)=\sup\{g_{K_t}^+(K_t\setminus K_s)-g_{K_t}^-(K_t\setminus K_s):0\leq s<t\leq T,t-s\leq\delta\}
\end{equation}
Clearly, $\omega(\delta)$ is increasing.

Next, we will prove the inequality in (\ref{eqn:3610.3}). Let $0\leq s'<t\leq T$ and $\delta'=t-s'$. Lemma \ref{lem:B} and the corollary to Lemma \ref{lem:C} show that there exists $s'\leq s<t$ with $K_s\cap\partial_{\mathbb{H}} K_t\not=\emptyset$ and
\begin{equation}\label{eqn:3610.1}
    |g_{K_t\cup L}(p_t)-g_{K_{s'}\cup L}(p_{s'})|
    =|\lambda(t)-\lambda(s')|
    \leq|\lambda(t)-\lambda(s)|
    =|g_{K_t\cup L}(p_t)-g_{K_s\cup L}(p_s)|
\end{equation}
Let $\delta=t-s\leq\delta'$, so $\omega(\delta)\leq\omega(\delta')$. Since $K_s\cap\partial_{\mathbb{H}} K_t\not=\emptyset$, by Lemma \ref{lem:A} we have for $r\in[s,t]$
\begin{equation}
    g_{K_t\cup L}^-(K_t\setminus K_s)\leq\lambda(r)\leq g_{K_t\cup L}^+(K_t\setminus K_s).
\end{equation}
So,
\begin{equation}
    |g_{K_t\cup L}(p_t)-g_{K_s\cup L}(p_s)|
    =|\lambda(t)-\lambda(s)|
    \leq g_{K_t\cup L}^+(K_t\setminus K_s)-g_{K_t\cup L}^-(K_t\setminus K_s).
\end{equation}
Since $g_{K_t\cup L}=g_{g_{K_t}(L)}\circ g_{K_t}$, Lemma \ref{lem:3.2.3} (b) shows
\begin{equation}\label{eqn:3610.2}
    g_{K_t\cup L}^+(K_t\setminus K_s)-g_{K_t\cup L}^-(K_t\setminus K_s)
    \leq g_{K_t}^+(K_t\setminus K_s)-g_{K_t}^-(K_t\setminus K_s)
    \leq\omega(\delta).
\end{equation}
Combining (\ref{eqn:3610.1}), (\ref{eqn:3610.2}), and $\omega(\delta)\leq\omega(\delta')$ gives the result.
\end{proof}

\begin{lemma}[3.8 \cite{RS14}]\label{lem:3.6.11}
Let $(K_t)_{t\in[0,T]}$ and $(L_v)_{v\in[0,S]}$ be two disjoint Loewner families. Then there exists constants $c,M>0$ and $\omega:[0,T]\to[0,\infty)$ increasing with $\lim_{\delta\downarrow 0}\omega(\delta)=\omega(0)=0$ such that
\begin{align}
    |g_{K_t\cup L_v}(p_t)-g_{K_s\cup L_u}(p_s)|
    &\leq\omega\left(\frac{1}{c}|\hcap(K_t\cup L_v)-\hcap(K_s\cup L_u)|\right)\\
    &\phantom{1}\qquad+\frac{M}{c}|\hcap(K_t\cup L_v)-\hcap(K_s\cup L_u)|
\end{align}
for all $s,t\in[0,T]$ and $u,v\in[0,S]$, where $p_t$ and $p_s$ are the prime ends corresponding to $\lambda(t)$ and $\lambda(s)$, respectively.
\end{lemma}

\begin{thm}[2.2 \cite{RS14}]\label{thm:3.6.5}
Let $A$ be a multi-Loewner hull with $\hcap(A)=2T$. For any Loewner parameterization $K_t=(K_t^1,K_t^2)$ of $A$, let $\lambda^j_K$ be the driving function of $K_t^j$ for $j=1,2$. Then the sets
\begin{equation}
    \{\lambda_K^j:[0,T]\to\mathbb{R}|K\text{ Loewner parameterization of }A\}
\end{equation}
are precompact subsets of the Banach space $C([0,T],\mathbb{R})$ for $j=1,2$.
\end{thm}

The first step in proving this theorem in \cite{RS14} is to get a uniform bound (in time) on $\lambda_K^j(t)$ for $j=1,2$. This bound, in our case, is
\begin{equation}
    g_{A}^-(A)
    = g_T^-(A)
    \leq\lambda_K^j(t)
    \leq g_T^+(A)
    = g_{A}^+(A).
\end{equation}
The rest of the proof in \cite{RS14} generalizes.

\begin{thm:3.6.2}[1.1 \cite{RS14}]
Let $K^1,...,K^n$ be disjoint Loewner hulls. Let $\hcap(K^1\cup\cdots\cup K^n)=2T$. Then there exist constants $w_1,...,w_n\in(0,1)$ with $\sum_{k=1}^{n}w_k=1$ and continuous driving functions $\lambda_1,...,\lambda_n:[0,T]\to\mathbb{R}$ so that
\begin{equation}
    \partial_t g_t(z)=\sum_{k=1}^{n}\frac{2w_k}{g_t(z)-\lambda_k(t)},\quad g_0(z)=z
\end{equation}
satisfies $g_T=g_{K^1\cup\cdots\cup K^n}$.
\end{thm:3.6.2}

\noindent The proof of this theorem is the proof in \cite{RS14}, but we include it so that the reader can see where the previously proven lemmas are used.

\begin{proof}
Let $K^1,K^2$ be disjoint Loewner hulls, $\hcap(K^1\cup K^2)=2$, $c_j=\frac{1}{2}\hcap(K_j)$.

Define $\alpha_{n,w}:[0,1]\to\{0,1\}$ for $(n,w)\in\mathbb{N}\times[0,1]$ as follows:
\begin{equation}
    \alpha_{n,w}(t)=\left\{
    \begin{array}{rr}
    1\qquad&t\in(\frac{k}{2^n},\frac{k+w}{2^n})\\
    0\qquad&t\in(\frac{k+w}{2^n}\frac{k+1}{2^n})
    \end{array}\right.
\end{equation}
for $k\in\{0,...,2^n\}$. Let
\begin{equation}
    \partial_t g_{t,n}(z)
    =\frac{2\alpha_{n,w}(t)}{g_{t,n}(z)-\lambda_{1,n}(t)}
    +\frac{2(1-\alpha_{n,w}(t))}{g_{t,n}(z)-\lambda_{2,n}(t)},\qquad
    g_0(z)=z.
\end{equation}

By the construction of $\alpha_{n,w}$ only one hull grows at a time. So, the Loewner equation (with a single driving function) gives that $\lambda_{1,n}(t)$ is defined on $\bigcup_{k=0}^{2^n-1}(\frac{k}{2^n},\frac{k+w}{2^n})$ (similarly for $\lambda_{2,n}(t)$). The disjointness of the hulls gives that we can extend $\lambda_{j,n}$ to be the image of $\lambda_{j,n}(t)$ under the map corresponding to the other hull. So, $\lambda_{1,n}$ and $\lambda_{2,n}$ are continuous on $[0,1]$. For $t\in[0,1]$ the hull at time $t$ is
\begin{equation}
    H_{n,w,t}=K^1_{x_{n,w,t}}\cup K^2_{y_{n,w,t}}
\end{equation}
where $x_{n,w,t}\in[0,1]$ depends continuously on $w$. For all $n\in\mathbb{N}$, $x_{n,0,1}=0$ and $x_{n,1,1}=1$ (as $w=0$ and $w=1$ correspond to single hull growth of $K^1$ and $K^2$ respectively). By the Intermediate Value Theorem, for each $n\in\mathbb{N}$ there exists $w_n$ so that $x_{n,w_n,1}=c_1$. By Lemma \ref{lem:3.2.2} (b), $y_{n,w_n,1}=c_2$. So, $H_{n,w_n,1}=K^1\cup K^2$. Which means that $\alpha_{n,w_n}$ is a sequence of weights and $\lambda_{j,n}$ are sequences of continuous driving function generating $K^1\cup K^2$.

By Theorem \ref{thm:3.6.5}, there is a subsequence of $\lambda_{1,n}$ converging to a function $\lambda_1$. Using Theorem \ref{thm:3.6.5} again on the corresponding subsequence of $\lambda_{2,n}$ we get that there is a further subsequence converging to a function $\lambda_2$. Furthermore, the corresponding subsequence of $w_n$ has a convergent subsequence converging to $w\in[0,1]$. We will now reindex this sequence by $n\in\mathbb{N}$.

Let
\begin{equation}
    \partial_t g_t(z)
    =\frac{2w}{g_{t,n}(z)-\lambda_{1,n}(t)}
    +\frac{2(1-w)}{g_{t,n}(z)-\lambda_{1,n}(t)},\qquad
    g_0(z)=z.
\end{equation}
Then it is easy to see that $\alpha_{n,w_n}$ converges weakly to $w$ in $L^1([0,1])$ (similar to Lemma \ref{lem:onehalfweights}). Now, by Theorem \ref{thm:rs2.4rrodf}, we have the result.
\end{proof}

\bibliographystyle{alpha}
\bibliography{biblio}

\end{document}